\documentclass[12pt]{amsart}  

\usepackage{custom_preamble} 

\begin{document}

\title{The Erd{\H o}s-Moser sum-free set problem}

\author{\tsname}
\address{\tsaddress}
\email{\tsemail}

\begin{abstract}
We show that there is an absolute $c>0$ such that if $A$ is a finite set of integers then there is a set $S \subset A$ of size at least $\log_3^{1+c}|A|$ such that the restricted sumset $\{s+s': s,s' \in S \text{ and } s \neq s'\}$ is disjoint from $A$.
\end{abstract}

\maketitle

\section{Introduction}

In this paper we are interested in some problems in additive combinatorics.  The standard introduction to this area is the book \cite{taovu::} by Tao and Vu and we have tried to give references to this book where possible.

Given finite sets $A$ and $S$ in an abelian group we write $A+S$ for the \textbf{sumset} $\{a+s: a \in A \text{ and }s \in S\}$ and $A \wh{+} S$ for the \textbf{restricted sumset} $\{a+s: a\in A, s \in S \text{ and } a \neq s\}$.  Erd{\H o}s \cite[p187]{erd::0} describes joint work with Moser in which they investigate the following question: if $A$ is a finite set of integers then what is the size of the largest set $S \subset A$ such that $(S \wh{+} S)\cap A=\emptyset$? They consider the restricted sumset to make the problem non-trivial: if $A$ is a set of consecutive powers of $2$ and $S \subset A$ has $(S+S) \cap A =\emptyset $ then $S$ has size at most $1$.

To discuss the problem we make a definition: for a finite set $A$ of integers define
\begin{equation*}
M(A):=\max\{|S|: S \subset A \text{ and } (S \wh{+} S) \cap A =\emptyset\}.
\end{equation*}
We are interested in lower bounds on $M(A)$ that are uniform in the size of $A$.  We record the history relevant to our interests below.  \cite[\S6.2.1]{taovu::} contains more details and a full survey can be found in \cite{taovu::1}.

In \cite{erd::0} a simple example is given to show (that for any natural number $N$ there is a set $A$ of size $N$ such) that $M(A) \leq \frac{1}{3}|A| +O(1)$ and this was improved, first by Selfridge \cite[p187]{erd::0}, then by Choi \cite[p190]{erd::0}, and then more substantially by Choi \cite[(2)]{cho::0} where it is shown that $M(A)\leq |A|^{2/5+o(1)}$.  The $o(1)$-term was refined by Baltz, Schoen, and Srivastav in \cite[Corollary 3]{balschsri::}\footnote{Equivalently \cite[Corollary 1]{balschsri::0}.} before Ruzsa adapted a classical construction of Behrend \cite{beh::} in \cite[Theorem]{ruz::04} to show the following.
\begin{theorem}[Ruzsa]\label{thm.ruz}  Given a natural number there is a set $A$ of that size such that
\begin{equation*}
M(A) = \exp(O(\sqrt{\log |A|})).
\end{equation*}
\end{theorem}
In the other direction, Erd{\H o}s and Moser showed that $M(A) \rightarrow \infty$ as $|A| \rightarrow \infty$, and Klarner showed that $M(A)=\Omega(\log |A|)$ (both results are mentioned on \cite[p187]{erd::0} though the proofs, or at least Klarner's, seem to have been lost \cite[{$\dagger$}, p630]{cho::0}).  Ruzsa showed that $M(A)> 2\log_3 |A| -1$ in \cite[Theorem]{ruz::04} by a greedy algorithm, and then Sudakov, Szemer{\'e}di and Vu made an important breakthrough in \cite{sudszevu::} giving the first super-logarithmic lower bound on $M(A)$ in \cite[Theorem 1.1]{sudszevu::}.  Their argument was improved by Dousse \cite[\S4]{dou::}, and then Shao \cite[Corollary 1.3]{sha::0} who showed that $M(A)=(\log \log |A|)^{\frac{1}{2}-o(1)}\log |A|$.  We shall show the following.
\begin{theorem}\label{thm.main}
For every finite set of integers $A$ we have
\begin{equation*}
M(A) = \log^{1+\Omega(1)} |A|.
\end{equation*}
\end{theorem}
The result as recorded in the abstract follows immediately from this.\footnote{For $|A|\in \{1,2,3\}$ it is trivial since $\log_3|A| \leq 1$ and any $S \subset A$ of size $1$ has $S\wh{+}S=\emptyset$.  On the other hand Theorem \ref{thm.main} immediately gives the result in the abstract for $|A| \geq C$ for some absolute $C>0$.  Finally, for $4 \leq |A| <C$ \cite[Theorem]{ruz::04} shows that
\begin{equation*}
M(A) > 2\log_3 |A| -1 > \left(1+\frac{\log 4- \log 3}{\log 4}\right)\log_3 |A| > \exp\left(\frac{1}{6}\right)\log_3 |A| > \log_3^{1+\frac{1}{6\log\log_3C}}|A|.
\end{equation*}
The result is proved.  (Note if $A=\{-1,0,1\}$ then $M(A)=1$ and so the best lower bound on $M(A)$ over all size $3$ sets is $1$, which is the reason for taking logarithms to the base $3$.)}

Our argument has two parts: the first makes use of specific properties of the integers and is dealt with in \S\ref{sec.over}.  The second part does not (at least it works in any abelian group with no $2$-torsion), and is covered in the remainder of the paper from \S\ref{sec.mod} onwards.  Indeed, because the integers do not play a significant role we are able to give a model argument in \S\ref{sec.mod} and we hope the reader familiar with the area will be able to understand the main ideas of our proof from \S\S\ref{sec.over}\&\ref{sec.mod} alone.

Our approach falls within the general strategy proposed by \cite{sudszevu::} so we begin in the next section with an overview of that, which also serves to explain how the improvements of Dousse and Shao arise. 

\section{Overview of the Sudakov-Szemer{\'e}di-Vu strategy}\label{sec.ssv}

Given sets $A$ and $X$ in an abelian group we say that $A$ is \textbf{$(k,X)$-summing}\footnote{Our terminology is not standard.  In \cite{sudszevu::} the authors say a set $S$ is \textbf{sum-free with respect to} $X$ if $(S \wh{+} S) \cap X =\emptyset$.  In \cite{taovu::2} the authors say that a set $S$ is \textbf{summing in} $X$ for the same thing.} if for any set $S \subset A$ with $|S| \geq k$ we have $(S \wh{+} S) \cap X \neq \emptyset$.  (We shall always take $k \geq 2$.)

The following proposition is the focus of the Sudakov-Szemer{\'e}di-Vu strategy.
\begin{proposition}\label{prop.key}
Suppose that $A\subset X \subset \Z$ have $|X| \leq (1+\eta)|A|$, and $A$ is $(k,X)$-summing for some $k \in \N$.  Then either $\eta = k^{-O(1)}$ or $|A| \leq F(k)$ for some universal (monotonically increasing) function $F:\N \rightarrow \N$.
\end{proposition}
\cite[Theorem 1.2]{sudszevu::} says we may take $F(k)=\exp(\exp(\exp(\exp(\exp(O(k))))))$.

Applying Proposition \ref{prop.key} with $X=A$  gives us that $M(A)\rightarrow \infty$ as $|A|\rightarrow \infty$, but the point of the proposition is the extra flexibility afforded by being able to take $X$ to be a little larger than $A$.  This means that it can be bootstrapped to give \cite[Theorem 1.1]{sudszevu::} (as is done in \cite[\S2]{sudszevu::}, and as we will do in \S\ref{sec.over}), and in general, given $F$, one gets
\begin{equation}\label{eqn.ssv}
M(A)=\Omega\left(\frac{F^{-1}(|A|)}{\log F^{-1}(|A|)} \log |A|\right).
\end{equation}
Many tools in additive combinatorics do not distinguish between different abelian groups and so to make use of them one needs to be in a situation where the conclusion does not depend on the underlying group.  In this case, if we try to replace the integers in Proposition \ref{prop.key} with a general abelian group we run into the problem that $A$ might be a subgroup.  In some sense this is the only obstacle as shown by Tao and Vu:
\begin{theorem}[{\cite[Theorem 1.2]{taovu::2}}]\label{thm.tv}
Suppose that $G$ is an abelian group and $A \subset G$ is finite and $(k,A)$-summing.  Then there is an integer $m \leq k$ and subgroups $H_1,\dots,H_m \leq G$ such that $|A \setminus (H_1\cup \cdots \cup H_m)| =O_k(1)$ and $|A\cap H_i| = \Omega_k(|H_i|)$ for $1 \leq i \leq m$.
\end{theorem}

The order property of the integers (not available in general abelian groups) is used in the derivation of (\ref{eqn.ssv}) from Proposition \ref{prop.key}, and the integers are also used essentially in Lemma \ref{lem.52} (and hence Lemma \ref{lem.comb}), but all other uses in this discussion section are for convenience.

Sudakov, Szemer{\'e}di and Vu capture a property of the integers that eliminates the subgroup examples of Theorem \ref{thm.tv} in the next lemma for which we require a definition.  Given a set $A$ in an abelian group $G$ the \textbf{additive energy of $A$} is defined\footnote{See  \cite[Definition 2.8]{taovu::} for a discussion.} to be
\begin{equation*}
E(A):=\|1_A \ast 1_{-A}\|_{\ell_2(G)}^2= \sum_z{\left(\sum_y{1_A(y)1_{-A}(z-y)}\right)^2}.
\end{equation*}
\begin{lemma}\label{lem.52}
Suppose that $X \subset \Z$ has $E(X) \geq \eta|X|^3$.  Then there is a set $X' \subset X$ such that $|X'| \geq \eta^{O(1)}|X|$ and $(2\cdot X') \cap X = \emptyset$.
\end{lemma}
The proof of this is \cite[Lemma 5.2]{sudszevu::} coupled with the Balog-Szemer{\'e}di-Gowers Theorem\footnote{In its usual form, which corresponds to \cite[Theorem 2.31((i) {$\Rightarrow$} (iv))]{taovu::} and then \cite[Exercise 2.3.15]{taovu::}).}, but we only need the statement for our discussion so do not record the details.  The rough idea (which we shall use to prove Lemma \ref{lem.int}) is to partition $\Z$ into sets $T_i:=\{x \in \Z: 2^i \divides x \text{ and } 2^{i+1} \not \divides x\}$ for $0 \leq i \leq \infty$.  There cannot be a lot of additive quadruples with each element in a different $T_i$ -- we then essentially take the largest index $i$ where the intersection with $X$ is not too small.

It turns out that a short application of Tur{\'a}n's theorem from graph theory -- essentially \cite[Lemma 3.1]{sudszevu::} -- shows that if $A$ is $(k,X)$-summing then $A$ has large additive energy.  One can think of this as saying a large part of $A$ is highly structured.
\begin{lemma}\label{lem.i}
Suppose that $G$ is an abelian group and $A,X \subset G$ are such that $|X| \leq K|A|$ and $A$ is $(k,X)$-summing.  Then either $|A|=k^{O(1)}$ or $E(A)=(kK)^{-O(1)}|A|^3$.
\end{lemma}
The above lemma is a consequence of Lemma \ref{lem.hen} proved later, though the method of \cite[Lemma 3.1]{sudszevu::} gives better constants for the $O(1)$-terms; again we only need the above form for the discussion.

In fact Lemma \ref{lem.hen} is importantly stronger than Lemma \ref{lem.i} and tells us that if $A$ is $(k,X)$-summing then either $A$ is small or else every subset of $A$ that is not small has large energy.  We formalise this in \S\ref{sec.over} in the notion of `hereditarily energetic'.  With this additional fact Lemmas \ref{lem.52} and \ref{lem.i} can be combined to give the following.
\begin{lemma}\label{lem.comb}
Suppose that $A\subset X \subset \Z$ are such that $|X| \leq (1+\eta)|A|$ and $A$ is $(k,X)$-summing.  Then either $|A|=k^{O(1)}$; or $\eta \geq k^{-O(1)}$; or there is a set $A' \subset A$ with $|A'|  \geq k^{-O(1)}|A|$ such that $(2\cdot A') \cap X = \emptyset$ and $E(A') \geq k^{-O(1)}|A'|^3$.
\end{lemma}
Again we omit the details as we only need the statement for discussion.  The key point is that if we are not in the first two outcomes, then after applying Lemma \ref{lem.i} to get that $A$ has large additive energy we apply Lemma \ref{lem.52} to $X$ (which inherits large additive energy from $A$), and since $\eta$ is small enough the set $X'$ in Lemma \ref{lem.52} is large enough that it necessarily has large intersection with $A$.

Although it was more unusual at the time of \cite{sudszevu::}, it is now common-place to apply the Balog-Szemer{\'e}di-Gowers-Freiman machinery in this sort of situation.   This tells us that
if $A \subset \Z$ and $E(A) \geq \eta |A|^3$, then there is an arithmetic progression $P$ such that
\begin{equation}\label{eqn.f}
|P| \geq |A|^{\eta^{o(1)}} \text{ and }|A \cap P| \geq \eta^{O(1)}|P|.
\end{equation}
This result with the $o(1)$-term replaced by $O(1)$ follows from \cite[Theorem 2.29]{taovu::} combined with \cite[Theorem 5.32]{taovu::}; the stronger bounds require the replacement of \cite[Theorem 2.29]{taovu::} by the improved estimates of Schoen \cite{sch::1}.

For our purposes arithmetic progressions are the same as intervals and the final ingredient we need is the following.
\begin{proposition}\label{prop.ss}
Suppose that $A \subset \{1,\dots,N\}$ has size $\alpha N$ and there is no proper $k$-tuple $(a_1,\dots,a_k) \in A^k$ with $a_i +a_j \in 2\cdot A$ for all $i<j$.  Then $N \leq F'(\alpha,k)$ for some universal function $F':(0,1]\times \N \rightarrow \N$ (decreasing in the first coordinate and increasing in the second).
\end{proposition}
To see why this is enough, apply the proposition to the output $A'$ of Lemma \ref{lem.comb} after applying (\ref{eqn.f}).  If $S \subset A'$ has $s+s' \in 2\cdot A'$ for all $s \neq s' \in S$ then since $(2\cdot A') \cap X = \emptyset$ we have that $(S\wh{+}S)\cap X = \emptyset$ and so $A'$ is not $(k,X)$-summing.  It follows that we can take
\begin{equation}\label{eqn.fdep}
F(k)=F'(k^{-O(1)},k)^{k^{-o(1)}}.
\end{equation}
in Proposition \ref{prop.key}.

Sudakov, Szemer{\'e}di and Vu proved in \cite[Corollary 3.3]{sudszevu::} that
\begin{equation*}
F'(\alpha,k) \leq \exp\left(\exp\left(\alpha^{-\exp(\exp(O(k)))}\right)\right)
\end{equation*}
by using Gowers' bounds for Szemer{\'e}di's Theorem.  Dousse noted that the system being counted in Proposition \ref{prop.ss} has complexity $1$ (in the sense of \cite[Definition 1.5]{gretao::7}) so one can study the configurations in Proposition \ref{prop.ss} using Fourier analysis rather than the higher order analogues of Gowers.  This is cheaper and she proved \cite[Corollary 3.4]{dou::} that
\begin{equation*}
F'(\alpha,k) \leq \exp\left(\exp\left(\alpha^{-O(k^2)}\right)\right).
\end{equation*}
Finally Shao implemented Dousse's Fourier argument in Bohr sets (following Bourgain \cite{bou::5} for three-term progressions which is the special case $k=2$ of Proposition \ref{prop.ss}; see \cite[Theorem 10.29]{taovu::} for an exposition) to show \cite[Theorem 1.3]{sha::0} that
\begin{equation*}
F'(\alpha,k) \leq \exp\left(\alpha^{-O(k^2)}\right).
\end{equation*}
(In fact this does not quite do justice to Shao's work: Dousse used a weaker version of the Balog-Szemer{\'e}di-Gowers-Freiman machinery which arises from applying a version of Freiman's theorem providing a progression containing the whole of the set $A$, rather than a progression inside $2A-2A$.  Shao noted that Ruzsa's Embedding Lemma \cite[Lemma 5.26]{taovu::} suffices and gives better bounds, although those improvements do not impact the level of our discussion above.)

It is natural to ask what sort of bounds we can expect on $F'$.  The worst example of bad $\alpha$-dependence in $F'$ comes from Behrend's construction \cite{beh::} (see \cite{elk::} and \cite{grewol::} for the state of the art) which tells us that $F'(\alpha,2) = \exp(\Omega(\log^2\alpha^{-1}))$.  On the other hand if we choose $A \subset \{1,\dots,N\}$ by selecting elements independently with probability $\alpha$, then\footnote{Any $S \subset \{1,\dots,N\}$ has $\P(S \wh{+}S \subset 2\cdot A) \leq \alpha^{|S\wh{+}S|}$.  Moreover if $S$ has size $k$ then $|S\wh{+}S| \geq 2k-3$, as can be seen by writing $S=\{s_1<\dots<s_k\}$ and noting that $s_1+s_2,\dots,s_1+s_k,s_k+s_2,\dots,s_k+s_{k-1}$ are distinct elements.  By embedding $\{1,\dots,N\}$ in $\Z/2N\Z$ and applying \cite[Proposition 23]{gre::07} it follows that
\begin{align*}
&\E{|\{S \subset A: |S|=k\text{ and } (S\wh{+}S) \subset 2\cdot A\}|}\\
 & \qquad \qquad \qquad \qquad \leq \E{|\{S \subset \{1,\dots,N\}: |S|=k \text{ and }(S\wh{+}S) \subset 2\cdot A\}|} \\
& \qquad \qquad \qquad \qquad \leq \sum_{m=2k-3}^{\binom{k}{2}}{\alpha^m|\{S \subset \{1,\dots,N\}: |S|=k \text{ and }|S\wh{+}S| =m\}|}\\
& \qquad \qquad \qquad \qquad \leq \sum_{m=2k-3}^{\binom{k}{2}}{\alpha^m(2N)^{1+O\left(\frac{m}{k}\right)}O\left(\frac{m}{k}\right)^{O(k)}}.
\end{align*}
If $k \geq C\log N$ for some absolute $C>0$ sufficiently large then each term in this last sum is $\alpha^m\exp(O(k\log mk^{-1}))=\exp(m(O(1)-\log\alpha^{-1}))$ from which it follows that we can choose $\alpha=\Omega(1)$ such that the expectation is strictly less than $\frac{1}{2}$.  On the other hand for $N$ sufficiently large $\P(|A| \geq \frac{1}{2}\alpha N)>\frac{1}{2}$, and it follows that there is some set $A$ with the claimed property.} there is a choice of $\alpha=\Omega(1)$ such that any $S \subset A$ with $S\wh{+}S \subset 2\cdot A$ has $|S| =O(\log N)$, from which it follows that $F'(\Omega(1),k)=\exp(\Omega(k))$.  By monotonicity of $F'$ in each of its variables we conclude that
\begin{equation}\label{eqn.lower}
F'(\alpha,k)=\exp(\Omega(k+\log^2\alpha^{-1})),
\end{equation}
and it is a natural question to ask if one can do better.

Although we may not have given the best combination of examples above, as far as we know it might be that $F'(\alpha,k)=\exp(O(k\alpha^{-o(1)}))$.  This would imply a considerable strengthening of Roth's theorem\footnote{The current best bounds there are due to Bloom \cite{blo::0} and imply that $F'(\alpha,2) \leq \exp(\alpha^{-1-o(1)})$.} and also that $F(k)\leq \exp(k^{1+o(1)})$ which in turn would give $M(A)\geq \log^{2-o(1)}|A|$ improving Theorem \ref{thm.main}.  Any bound of the form $F'(\alpha,k)=\exp((k\alpha^{-1})^{O(1)})$ already leads to $F(k)=\exp(k^{O(1)})$ and a different proof of Theorem \ref{thm.main}.  That being said it is not completely clear how to get a singly rather than doubly exponential dependence on $k$ in $F'$ since Fourier methods seem to rely on regularising a set that is exponentially small in $k$; perhaps a first step of showing $F'(\alpha,k) \leq \exp(\alpha^{-O(k)})$ is within reach of those methods.

We do not prove Theorem \ref{thm.main} by proving better bounds for $F'$.  Our advantage comes from one weakness of the above: all we are looking for is sums from $A$ that are in $A^c$, in principle a much less demanding condition (at least if $A$ is thin) than that in Proposition \ref{prop.ss} where we ask that they are in $2\cdot A$.  We shall explain this further in \S\ref{sec.mod}.

The aim of the remainder of the paper is to prove the following quantitative version of Proposition \ref{prop.key}.
\begin{proposition}\label{prop.key2}
Suppose that $A\subset X \subset \Z$; $|X| \leq (1+\eta)|A|$; and $A$ is $(k,X)$-summing for some $k \in \N$.  Then either $\eta = k^{-O(1)}$; or $|A| \leq\exp(k^{C+o(1)})$ for some absolute $C>0$.
\end{proposition}
A value of $C$ could be calculated from our work but there is considerable scope for optimising it.  We have tried to indicate some places this might be possible, but it also seems likely that some of the convenient decoupling of the argument (for example the introduction of hereditarily-energetic sets in the next section) might be lost in a very careful optimisation.

It seems quite possible that $C=1$ is achievable, though it is not clear one can expect to do better without a new idea, and in light of (\ref{eqn.lower}) one cannot do better by trying to improve $F'$ in Proposition \ref{prop.ss}.  Given this it seems to us that the natural next question is whether or not $M(A) = \log^{2+\Omega(1)}|A|$.

\section{Overview of our argument}\label{sec.over}

In this section we decouple our arguments into those that require order and divisibility properties of the integers and those that in some sense do not.  The latter are captured by Proposition \ref{prop.g} below; everything else is covered in the present section.

We begin by using the order structure on the integers to prove Theorem \ref{thm.main} using Proposition \ref{prop.key2}.  The argument is included for completeness; it is essentially\footnote{Our argument is slightly sloppier which we can afford because of the strength of Proposition \ref{prop.key2}.} the same as the bootstrapping of \cite[Theorem 1.2]{sudszevu::} to get \cite[Theorem 1.1]{sudszevu::} in \cite[\S2]{sudszevu::}.

\begin{proof}[Proof of Theorem \ref{thm.main}]
Let $k,m \in \N$ be parameters to be optimised; let $l \in \N$ be such that $l=k^{O(1)}$ and if $\eta \geq l^{-1}$ then the first conclusion of Proposition \ref{prop.key2} does not hold; and let $D=\exp(k^{O(1)})$ be a natural number such that if $|A| \geq D$ then the second conclusion does not hold. 

We construct sets iteratively as follows: let $m \in \N_0$ be such that $2D(2l(mk+1))^m \leq |A|$, and for $0 \leq i \leq m$ let $Z_i$ be the $2D(2l(mk+1))^i$ largest elements of $A$ so that $|Z_{i}| \geq 2l(mk+1)|Z_{i-1}|$.  For $0 \leq r < m$ we shall define sets $S_0,\dots,S_r$ of size $k$ such that $(S_i \wh{+} S_i )\cap Z_i = \emptyset$ and an auxiliary sequence of $A_i$s with $S_i \subset A_i$ and
\begin{equation*}
A_0:=Z_0 \text{ and }A_{i+1}:=Z_{i+1} \setminus \left(Z_{i}-\left(\{0\}\cup\bigcup_{j\leq i}{S_j}\right)\right) \text{ for }0 \leq i \leq r.
\end{equation*}
For clarity we note that $\setminus$ denotes relative complement of sets here and $-$ denotes the difference of two set.  If $i<m$ then
\begin{align*}
|Z_{i+1}| & \leq |A_{i+1}| + |Z_{i}|\left(1+\sum_{j\leq i}{|S_j|}\right)\\ & \leq |A_{i+1}| + \frac{1}{2l(mk+1)}|Z_{i+1}|(1+mk) = |A_{i+1}| + \frac{1}{2l}|Z_{i+1}|.
\end{align*}
Rearranging we have that $|Z_{i+1}| \leq (1+l^{-1})|A_{i+1}|$ and by design $A_{i+1} \subset Z_{i+1}$. Moreover, $|A_{i+1}| \geq \frac{1}{2}|Z_{i+1}| \geq D$, and so it follows from Proposition \ref{prop.key2} that $A_{i+1}$ is not $(k,Z_{i+1})$-summing -- equivalently, there is a set $S_{i+1} \subset A_{i+1}$ with $(S_{i+1} \wh{+} S_{i+1} )\cap Z_{i+1} = \emptyset$ as required.

For $0 \leq r<m$ consider the set $S:=\bigcup_{i=0}^r{S_i}$.  Then
\begin{align*}
(S\wh{+}S)\cap A  & =\left( \left(\bigcup_{0 \leq i <j \leq r}{S_i+S_j}\right) \cup \left(\bigcup_{i=0}^r{ S_i \wh{+} S_i }\right)\right)\cap A\\
&= \left(\bigcup_{0 \leq i <j \leq r}{(S_i+S_j) \cap A}\right) \cup \left(\bigcup_{i=0}^r{ (S_i \wh{+} S_i)\cap A }\right).
\end{align*}
Suppose that $0 \leq i \leq r$.  Since $S_i \subset Z_i$ -- of the largest $2D(2l(mk+1))^i$ elements of $A$ -- and when two positive integers are added their size increases we see that $(S_i \wh{+} S_i)\cap A = (S_i \wh{+} S_i) \cap Z_i = \emptyset$.

Then suppose that $0 \leq i<j\leq r$.  Then by similar reasoning $(S_i+S_j)\cap A = (S_i + S_j)\cap Z_{j-1}$ and $S_j \cap (Z_{j-1}-S_i)=\emptyset$ by design so $(S_i+S_j)\cap A=\emptyset$.

Combining these we see that $(S\wh{+}S)\cap A=\emptyset$.  Since the sets $(A_i)_i$ are disjoint so are the $S_i$s, so $|S| \geq mk$, and our task is to maximise this subject to $2D(2l(mk+1))^m \leq |A|$.  We can certainly take $k=\log^{\Omega(1)}|A|$ and $2D \leq \sqrt{|A|}$ and $m = \Omega(\log |A|/\log \log |A|)$ such that $(2l(mk+1))^m \leq \sqrt{|A|}$ from which the result follows.
\end{proof}
It will be useful to have some notation for the Fourier transform.  This is developed in \cite[Chapter 4]{taovu::}, but we shall use different conventions.  Suppose that $G$ is an abelian group.  Given $f,g \in \ell_1(G)$ we write $f \ast g$ for the \textbf{convolution} of $f$ and $g$ defined point-wise by
\begin{equation*}
f \ast g(x):=\sum_z{f(z)g(x-z)} \text{ for all }x \in G.
\end{equation*}
We write $\wh{G}$ for the compact abelian group of characters on $G$ and define the Fourier transform of $f \in \ell_1(G)$ to be
\begin{equation*}
\wh{f}:\wh{G} \rightarrow \C; \gamma \mapsto \sum_{z \in G}{f(z)\overline{\gamma(z)}}.
\end{equation*}
The group $\wh{G}$ is endowed with a Haar probability measure in such a way that we have Plancherel's theorem (see \cite[Theorem 1.6.1]{rud::1}):
\begin{equation*}
\sum_x{f(x)^2} = \int{|\wh{f}(\gamma)|^2d\gamma} \text{ for all }f.
\end{equation*}

Sets with small doubling (see \cite[\S2.2]{taovu::}) and large additive energy \cite[Theorem 2.31]{taovu::} are staples of additive combinatorics.  We shall need an intermediate concept: we say that $A \subset G$ is \textbf{$\nu$-hereditarily energetic} if
\begin{equation*}
E(S) \geq \nu \sigma |S|^3 \text{ for all }S\subset A \text{ with }|S| \geq \sigma |A|.
\end{equation*}
If $A$ has $E(A) \geq \nu |A|^3$ and $S \subset A$ is chosen independently at random with probability $\sigma$ then typically
\begin{equation*}
|S| \approx \sigma |A| \text{ and }E(S) \gtrsim \nu \sigma^4|A|^3 \approx \nu \sigma |S|^3;
\end{equation*}
the definition says we never do worse than this. We take the name from the related concept of hereditarily non-uniform sets defined in \cite[\S2]{gre::0}, and the notion is implicit in numerous papers.

There are many examples including the $\alpha$-spectrum \cite[Definition 4.33]{taovu::}, which we shall not discuss\footnote{If a set has large additive energy then the large spectrum is large (for a suitable threshold) by Parseval's theorem.  The case $k=2$ of \cite[Theorem 5]{shk::4} then essentially shows that the large spectrum is hereditarily energetic in a similar way to Lemma \ref{lem.hed} (\ref{pt.4}).}, and symmetry sets, which we shall.  Recall, following \cite[Definition 2.32]{taovu::} that if $\nu \in (0,1]$ and $X \subset G$ then the \textbf{symmetry set of $X$ at threshold $\nu$} is the set
\begin{equation*}
\Sym_\nu(X):=\{x \in G: 1_X \ast 1_{-X} (x)>\nu |X|\}.
\end{equation*}
\begin{lemma}[Basic facts about hereditarily energetic sets]\label{lem.hed}\
\begin{enumerate}
\item\label{pt.1} \emph{(Sets with small doubling)} Suppose that $|A+A| \leq K|A|$.  Then $A$ is $K^{-1}$-hereditarily energetic.
\item\label{pt.2} \emph{(Monotonicity)} Suppose that $A$ is $\nu$-hereditarily energetic and $A' \subset A$ has size $\epsilon |A|$.  Then $A'$ is $\epsilon\nu$-hereditarily energetic.
\item\label{pt.3} \emph{(Unions)} Suppose that $A$ and $A'$ are $\nu$-hereditarily energetic.  Then $A \cup A'$ is $\Omega(\nu)$-hereditarily energetic.
\item\label{pt.4} \emph{(Symmetry sets)} Suppose that $E(A) \geq \nu |A|^3$.  Then $\Sym_{\frac{1}{2}\nu}(A)$ is $\Omega(\nu^{4})$-hereditarily energetic.
\end{enumerate}
\end{lemma}
\begin{proof}
For (\ref{pt.1}) suppose $S \subset A$ has $|S| \geq \sigma |A|$.  Then
\begin{equation*}
\|1_{S} \ast 1_{-S}\|_{\ell_2(G)}^2 \geq \frac{|S|^4}{|S+S|} \geq \frac{|S|^4}{|A+A|} \geq K^{-1}\sigma |S|^3,
\end{equation*}
and we have the claim.  (\ref{pt.2}) is trivial.  For (\ref{pt.3}) suppose that $S \subset A \cup A'$ has size $\sigma |A \cup A'|$.  Then without loss of generality $|S\cap A| \geq \frac{1}{2}\sigma |A\cup A'|$. It follows that
\begin{equation*}
E(S) \geq E(S\cap A) \geq \nu\frac{1}{|A|} |S\cap A|^4 \geq \frac{\nu}{8}\sigma |S|^3
\end{equation*}
as claimed.

Finally to prove (\ref{pt.4}) note that since $E(A) \geq \nu |A|^3$ we have $\left|\Sym_{\frac{1}{2}\nu}(A)\right| \geq \frac{1}{2}\nu |A|$.  Suppose that $S \subset \Sym_{\frac{1}{2}\nu}(A)$.  Then
\begin{align*}
E(S)|A|^5 \geq \|\wh{1_S}\|_{L_4(\wh{G})}^4\||\wh{1_A}|^2\|_{L_{\frac{4}{3}}(\wh{G})}^4& \geq \left|\langle \wh{1_S},|\wh{1_A}|^2\rangle_{L_2(\wh{G})}\right|^4 \\ &= \left|\langle 1_S,1_A \ast 1_{-A}\rangle_{\ell_2(G)}\right|^4 \geq \left(\frac{1}{2}\nu|S| |A|\right)^4.
\end{align*}
The final result follows.
\end{proof}
In view of (\ref{pt.1}) and (\ref{pt.3}) above, the union of two sets with small doubling is hereditarily energetic, but of course the union need not have small doubling so the notion is strictly weaker than having small doubling.  At the other end of the range, if $A$ is $\nu$-hereditarily energetic then it has large additive energy but the converse need not be true as can be seen by taking any set of large additive energy and adjoining a dissociated set of the same size.

We formulate the version of Lemma \ref{lem.i} we need as follows.
\begin{lemma}\label{lem.hen}
Suppose that $G$ is an abelian group and $A,X \subset G$ are such that $|X| \leq K|A|$ and $A$ is $(k,X)$-summing.  Then $A$ is $(kK)^{-O(1)}$-hereditarily energetic.
\end{lemma}
\begin{proof}
Suppose that $|A| \geq k^4$, and note that if $x \in A^k$ then there are two indices $i \neq j$ such that either $x_i=x_j$ or $x_i+x_j \in X$.  It follows that if $S \subset A$ we have
\begin{equation*}
\sum_{x \in S^k}{\sum_{i<j}{1_X(x_i+x_j)}} \geq |S|^k - \binom{k}{2}|S|^{k-1}.
\end{equation*}
Defining $\sigma$ by $|S|=\sigma |A|$ we conclude that either $|S| \leq k^2$, in which case $E(S) \geq |S|^2 \geq \sigma |S|^3$; or $|S| >k^2$ in which case
 \begin{equation*}
\binom{k}{2}|S|^{k-2}\langle 1_S \ast 1_S,1_X\rangle_{\ell_2(G)} =\sum_{x \in S^k}{\sum_{i<j}{1_X(x_i+x_j)}} \geq \frac{1}{2}|S|^k.
\end{equation*}
But then
\begin{equation*}
\frac{1}{k^2}|S|^2 \leq \langle 1_S \ast 1_S,1_X\rangle_{\ell_2(G)} = \int{\wh{1_S}(\gamma)^2\overline{\wh{1_X}(\gamma)}d\gamma} \leq E(S)^{\frac{1}{2}}|X|^{\frac{1}{2}},
\end{equation*}
and so
\begin{equation*}
E(S) \geq \frac{1}{k^4}\cdot \frac{|S|}{|X|}\cdot |S|^3 \geq \frac{1}{Kk^4}\sigma |S|^3.
\end{equation*}
The result is proved.
\end{proof}
We need a little notation.  We use the language of $2$-adic valuations but this is only for brevity.  Direct discussion of divisibility properties rather, say, than the use of the ultra-metric property, is very easy.

If $z \in \Z$ we write $|z|_2$ for the $2$-adic valuation of $z$, that is $|z|_2=2^{-i}$ where $2^i \divides z$ and $2^{i+1} \not \divides z$, with the convention that $|0|_2=0$.  For $A \subset \Z$ we write
\begin{equation*}
A_i:=\{z\in A: |z|_2=2^{-i}\} \text{ for all }i \in \N_0\cup\{\infty\}
\end{equation*}
(with the convention that $2^{-\infty}=0$).  Thus $\{A_i: i \in \N_0\cup \{\infty\}\}$ is a set of disjoint sets whose union is $A$.  For $\epsilon \in (0,1]$ we write
\begin{equation*}
I_\epsilon(A):=\{i \in \N_0\cup\{\infty\} : |A_i| > \epsilon |A|\} \text{ and } A_\epsilon:=\bigcup{\{A_i: i \in I_\epsilon(A)\}};
\end{equation*}
note that $|I_\epsilon(A)| < \epsilon^{-1}$.  This method of decomposing $A$ is used in the proof of \cite[Lemma 5.2]{sudszevu::}, and the following lemma is very much in the spirit of that result.
\begin{lemma}\label{lem.energy}
Suppose that $A \subset \Z$ is $\nu$-hereditarily energetic; and $\epsilon \in (0,1]$ is a parameter.  Then $|A\setminus A_{\epsilon^2\nu}| =O(\epsilon|A|)$.
\end{lemma}
\begin{proof}
Suppose that $(a_1,a_2,a_3,a_4) \in A^4$ are such that $a_1+a_2=a_3+a_4$.  Let $i$ be such that $|a_i|_2$ is maximal, and note that $|a_i| = |a_j+a_k-a_l|$ for any $j,k,l$ with $\{i,j,k,l\}=\{1,2,3,4\}$.  Since the $2$-adic valuation induces an ultra-metric $|a_i|_2 \leq \max\{|a_j|_2,|a_k|_2,|a_l|_2\} \leq |a_i|_2$, so there is some $j \in \{1,2,3,4\}$ with $j \neq i$ such that $|a_j|_2=|a_i|_2$.  Writing $A^-:=A \setminus A_{\epsilon^2\nu}$ it follows that
\begin{align*}
E(A^-) & \leq \sum_{i \not \in I_{\epsilon^2\nu}(A)}{\left(\langle 1_{A_i} \ast 1_{-A_i},1_{A^-}\ast 1_{-A^-}\rangle_{\ell_2(\Z)}+\langle 1_{A_i} \ast 1_{-A^-},1_{A_i}\ast 1_{-A^-}\rangle_{\ell_2(\Z)}\right. }\\
& \qquad \qquad +\langle 1_{A_i} \ast 1_{-A^-},1_{A^-}\ast 1_{-A_i}\rangle_{\ell_2(\Z)}+\langle 1_{A^-} \ast 1_{-A_i},1_{A_i}\ast 1_{-A^-}\rangle_{\ell_2(\Z)}\\
& \qquad \qquad \left.+\langle 1_{A^-} \ast 1_{-A_i},1_{A^-}\ast 1_{-A_i}\rangle_{\ell_2(\Z)}+\langle 1_{A^-} \ast 1_{-A^-},1_{A_i}\ast 1_{-A_i}\rangle_{\ell_2(\Z)}\right)\\
& \leq 6\sum_{i \not \in I_{\epsilon^2\nu}(A)}{|A_i|^2|A^-|} \leq 6\epsilon^2\nu |A||A^-| \sum_{i \not \in I_{\epsilon^2\nu}(A)}{|A_j|} = 6\epsilon^2\nu|A||A^-|^2.
\end{align*}
On the other hand the left hand side is at least $\nu |A^-|^4/|A|$ since $A$ is $\nu$-hereditarily energetic.  The result follows on rearranging.
\end{proof}
We shall use the above to get the following.
\begin{lemma}\label{lem.int}
Suppose that $A,S \subset \Z$ are both $\nu$-hereditarily energetic; $|S| \leq K|A|$; and $\kappa \in (0,1]$ and $r \in \N$ are parameters.  Then either $r\leq (\kappa^{-1}\nu^{-1}K)^{O(1)}$; or there is some $1\leq j \leq r$ such that $|(2^j\cdot S)\cap A| <\kappa |A|$.
\end{lemma}
\begin{proof}
Let $\epsilon = \Omega(\kappa K^{-1})$ be such that the $O(\epsilon)$-term in Lemma \ref{lem.energy} is at most $\frac{\kappa}{2(1+K)}$.  It follows that
\begin{equation*}
|A\setminus A_{\epsilon^2\nu}| \leq \frac{\kappa}{2}|A| \text{ and } |S\setminus S_{\epsilon^2\nu}| \leq \frac{\kappa}{2(1+K)}|S| < \frac{\kappa}{2}|A|.
\end{equation*}
Suppose that $|(2^j\cdot S)\cap A| \geq\kappa |A|$ for all $1\leq j \leq r$.  Then by the triangle inequality
\begin{equation*}
|(2^j\cdot S_{\epsilon^2\nu}) \cap A_{\epsilon^2\nu}| > \kappa |A| -\frac{\kappa}{2}|A|-\frac{\kappa}{2}|A|>0 \text{ for all }i \leq j \leq r.
\end{equation*}
Let $t \in (2^j\cdot S_{\epsilon^2\nu}) \cap A_{\epsilon^2\nu}$ then there are odd numbers $a$ and $s$ such that $t=2^ia$ and $t = 2^{j+i'}s$ with $i \in I_{\epsilon^2\nu}(A)$ and $i' \in I_{\epsilon^2\nu}(S)$.  It follows that $j=i-i'$.  However, $|I_{\epsilon^2\nu}(A)|<\epsilon^{-2}\nu^{-1}$ and similarly $| I_{\epsilon^2\nu}(S)| < \epsilon^{-2}\nu^{-1}$ and so $r \leq \epsilon^{-4}\nu^{-2}$ and the result follows.
\end{proof}

The main ingredient replacing Proposition \ref{prop.ss} and allied arguments is the following proposition.
\begin{proposition}\label{prop.g}
Suppose that $G$ has no $2$-torsion; $A,X \subset G$; $|X\setminus A| \leq \eta|A|$; $A$ is $(k,X)$-summing; and $r \in \N$ is a parameter.  Then either $\eta^{-1} = (kr)^{O(1)}$; or $|A| \leq \exp((kr)^{O(1)})$; or there is a $k^{-O(1)}$-hereditarily energetic set $S$ with $|S| \geq k^{-O(1)}|A|$ and $|(2^j \cdot S)\cap A| \geq k^{-O(1)}|S|$ for all $1 \leq j \leq r$.
\end{proposition}
There are a couple of remarks worth making.  First, the proof we give actually shows the stronger conclusion that $S$ has doubling $k^{O(1)}$.  We do not need this, but it may be worth noting and is probably helpful in understanding the structure of the proof.  The notion of hereditarily energetic is necessary for dealing with $(k,Z)$-summing sets (see Lemma \ref{lem.hen}) which need not have small doubling.

Secondly, it is not surprising that we ask for $G$ to have no $2$-torsion in view of the conclusion: no analogue of Proposition \ref{prop.g} can be true for $G=(\Z/2\Z)^n$ since the final conclusion collapses to $|A| \leq k^{O(1)}$ since whatever $S$ is, $2\cdot S = \{0_G\}$ in this case.  Thus if $A=X=G$ then $A$ is $(2,X)$-summing (and so $(k,X)$-summing) and we can take $\eta=0$, but there is no bound on $|A|$ in terms of $k$.

We shall first prove Proposition \ref{prop.g} in a model setting in \S\ref{sec.mod} to illustrate our arguments, before moving on to the general result.

With this in hand we are ready for the main result. It is worth giving a word of explanation.  A number of our arguments involve careful dependences between various parameters and we shall say things like `let $\epsilon_0$ be such that the first conclusion of Theorem X does not hold'.  When we say this the conclusions of Theorem X will begin with a number of inequalities between parameters and we shall want to choose things so that those inequalities do not hold leading to the more substantial conclusion(s) of the theorem.  The proof below while short will give a flavour; the later arguments are more involved.
\begin{proof}[Proof of Proposition \ref{prop.key2}]
By Lemma \ref{lem.hen} $A$ is $\nu_0=k^{-O(1)}$-hereditarily energetic (assuming, as we may, that $\eta\leq 1$).  Let $\nu_1 = k^{-O(1)}$ such that the set $S$ in Proposition \ref{prop.g} is always $\nu_1$-hereditarily energetic; $\sigma_0=k^{-O(1)}$ be such that $|S| \geq \sigma_0 |A|$; and $\tau_0=k^{-O(1)}$ be such that $|(2^j\cdot S) \cap A| \geq \tau_0|S|$.  Finally, let $r = k^{O(1)}$ be such that the first conclusion of Lemma \ref{lem.int} applied to sets that are $\min\{\nu_0,\nu_1\}$-hereditarily energetic with size ratio is at most $\sigma_0^{-1}$, and with parameters $\sigma_0\tau_0$ and $r$ does not hold.

Apply Proposition \ref{prop.g} with parameter $r$.  Then either $\eta^{-1}=k^{O(1)}$; or $|A| \leq \exp(k^{O(1)})$; or else there is a $\nu_1$-hereditarily energetic set $S$ with $|S| \geq \sigma_0 |A|$ and $|(2^j\cdot S)\cap A| \geq \tau_0|S| \geq \tau_0\sigma_0|A|$ for all $1 \leq j \leq r$.  By Lemma \ref{lem.int}, whose first conclusion does not hold by design, we get a contradiction and the result is proved.
\end{proof}

\section{A model for Proposition \ref{prop.g}}\label{sec.mod}

The finite field model is useful for illustrating arguments in arithmetic combinatorics without a lot of the technical difficulties involved in general abelian groups.  One of the earliest introductions to the model is in \cite{gre::9}, and a summary of more recent developments can be found in \cite{wol::3}.

Throughout this section $V$ denotes a vector space over a finite field $\F_p$ where $p$ is an odd prime.  Our aim is to prove the following model version of Proposition \ref{prop.g}.
\begin{proposition}\label{prop.modg}
Suppose that $A,X \subset V$; $|X\setminus A| \leq \eta |A|$ for some $\eta \in (0,1]$; $A$ is $(k,X)$-summing for some $k \geq 2$; and $r \in \N$ is a parameter.  Then either $\eta^{-1} \leq (rk)^{O(1)}$; or $|A| \leq p^{(kr)^{O(1)}}$; or there is a $k^{-O(1)}$-hereditarily energetic set $S$ with $|S| \geq k^{-O(1)}|A|$, and $|(2^j \cdot S)\cap A| \geq k^{-O(1)}|S|$ for all $1 \leq j \leq r$.
\end{proposition}
There are three qualifying remarks to make.  First, there is (necessarily) no analogue of Lemma \ref{lem.int} in $V$, and without such it is difficult to argue that the conclusion of Proposition \ref{prop.modg} is terribly significant.

Secondly, the fact that the bound on $|A|$ is dependent on $p$ may look odd to those not familiar with this sort of model since there is no equivalent in Proposition \ref{prop.g}.  This is standard for the model and arises because in this setting we use genuine subspaces rather than `approximate' subgroups.  The size of genuine subspaces is a power of $p$, and replacing an `approximate' subgroup with a genuine subspace usually involves shrinking by a factor that is a power of $p$.

Typically in the model we think of $p$ as fixed although this could be seen to be in conflict with the fact that the $r$-dependence is important.  Indeed, for any natural number $j \in (p,r]$ there will be some $j' \in \{1,\dots,p\}$ such that $2^j\equiv 2^{j'}\pmod p$ -- it follows that one might as well take $r \leq p$.  We shall not make this simplification as it is an artefact of the model.

Finally, it may be that a better result could be proved using the new polynomial techniques of Croot, Lev and Pach \cite{crolevpac::}.  We have not used their method because at present there is no known way to convert it to give arguments in the setting we ultimately need (the integers).

We begin with a sketch in which we aim to convey the structure of the argument.  The more detailed work afterwards is to explain where the bounds come from.

Before going into the sketch we recall the two sets of hypotheses in Proposition \ref{prop.modg} and mention where they arise: we have that $A$ is $(k,X)$-summing which is used in \ref{step1} and \ref{step5}; we also (more or less) have that $|X\setminus A| \leq \eta |A|$ for $\eta$ small in terms of $k$, which is used in \ref{step5}.

Given a non-empty set $S$ we write $m_S$ for the uniform probability measure supported on $S$.\\

\renewcommand{\theenumi}{\textbf{STEP \Roman{enumi}}}
\begin{enumerate}[wide, labelwidth=!, labelindent=0pt]
\setlength\itemsep{1em}
\item \label{stepdef} We study $A$ with respect to an average of uniform probability measures on translates of (possibly different) subspaces.  Formally, a \textbf{weighted cover of $S$ by subspaces}\footnote{We shall tend to drop the `by subspaces' part.} is a pair $(z,Z)$ where $z$ is a $V$-valued random variable, $Z$ is a finite-subspace-of-$V$-valued random variable, and $\E{m_{z+Z}}=m_S$. Since we may take $V$ to be finite we can assume these random variables take finitely many values so that there are no analytic issues to worry about.

We construct new weighted covers by specifying conditional joint distributions: given a weighted cover $(z,Z)$ of $S$, and a weighted cover $(w^{(x,U)},W^{(x,U)})$ of $x+U$ for each $x \in V$ and $U \leq V$ (finite), we can define a new weighted cover $(z',Z')$ of $S$ by specifying that
\begin{equation*}
\P(z'=x' \text{ and } Z'=U'|z,Z)=\P(w^{(z,Z)}=x' \text{ and } W^{(z,Z)}=U').
\end{equation*}
Indeed, we have
\begin{equation*}
\E{m_{z'+Z'}} = \E{(\E{(m_{z'+Z'}|z,Z)})} = \E{(\E{(m_{w^{(z,Z)}+W^{(z,Z)}}|z,Z)})} = \E{m_{z+Z}} = m_S.
\end{equation*}
When we need to refer to the underlying sample space we call it $\Omega$, and $\omega$ will always denote an element of $\Omega$.

Given a weighted cover $(z,Z)$ of $S$, we shall also need weighted covers of $2^j\cdot S$ and conversely.  In the model setting this is particularly easy: if $(z,Z)$ is a weighted cover of $S$ then $(2^jz,Z)$ is a weighted cover of $2^j\cdot S$ (since $2^j$ is just a scalar, and so $2^j\cdot (z+Z) = 2^jz +Z$) whenever $j\in \N_0$.  We do not get this as cheaply in the non-model setting.
\item \label{step1} First we use the Balog-Szemer{\'e}di-Gowers-Freiman machinery to find a set $S$ having large intersection with $A$, and a subspace $U$ of size not too much smaller than $A$, such that $m_{S} = m_S \ast m_U$.  One can also require that $S$ has small doubling and so it is hereditarily energetic.  Considering $(S,m_S)$ as a probability space and writing $z:S \rightarrow V$ for the natural inclusion, and $Z$ the constant function taking the value $U$ we have that $m_S=\E{m_{z+Z}}$ and so $(z,Z)$ is a weighted cover of $S$.  This is Lemma \ref{lem.freiman}.

\item \label{step3} Suppose that $(z,Z)$ is a weighted cover and $A$ is not highly uniform on $z(\omega)+Z(\omega)$ then a Fourier argument tells us that there is a large subspace $Z' \leq Z(\omega)$ such that $m_{z(\omega)+Z(\omega)}=\E_{z' \in z(\omega)+Z(\omega)}{m_{z'+Z'}}$ and
\begin{equation*}
\E_{z' \in z(\omega)+Z(\omega)}{m_{z'+Z'}(A)^2} >  m_{z(\omega)+Z(\omega)}(A)^2
\end{equation*}
where the size of the difference between left and right increases with the level of non-uniformity.

If there is a small (but not too small) proportion of $\omega \in \Omega$ with $A$ not highly uniform on $z(\omega)+Z(\omega)$ then the above can be used to produce a new weighted cover $(z',Z')$ of $S$ where
\begin{equation*}
\E{m_{z'+Z'}(A)^2} > \E{m_{z+Z}(A)^2}.
\end{equation*}
Again, the size of the difference is dependent on the notion of small above and the level of non-uniformity.  As a result this can be set in an iteration until we end up with a weighted cover where none of the subspaces are too small and where $A$ is highly uniform on $z(\omega)+Z(\omega)$ for almost all $\omega \in \Omega$. This is Lemma \ref{lem.unif}.  In the end we shall need $A$ to be highly uniform on  $2^jz(\omega) + Z(\omega)$ for almost all $\omega \in \Omega$ and all $0 \leq j \leq r$.  This can be done at a cost of iterating $r$ times as often.  This (combined with \ref{step1}) is Corollary \ref{cor.y}.
\item \label{step4} By averaging (using the fact that $(z,Z)$ is a weighted cover of $S$) most of the mass of $A$ in $S$ corresponds to $\omega$s where $m_{z(\omega)+Z(\omega)}(A)$ is not too small.  By \ref{step3}, $A$ is highly uniform on most of these $\omega$s and \emph{if} $m_{2z(\omega) + Z(\omega)}(X^c)$ is very large (meaning close to $1$) we can use this uniformity, the size of $A$ on $z(\omega)+Z(\omega)$, and the pigeonhole principle to show that there are many $z \in (A\cap(z(\omega)+Z(\omega)))^k$ such that $z_i +z_j \in A^c$ whenever $i<j$.  Crucially, counting with the pigeonhole principle requires much less uniformity than counting $z$s for which $z_i+z_j \in 2\cdot A$ as in Proposition \ref{prop.ss}. This counting is Lemma \ref{lem.ct}.
\item \label{step5} Assuming that $A$ is large enough that most $z \in (A\cap(z(\omega)+Z(\omega)))^k$ have $z_i \neq z_j$ for $i \neq j$, then since $A$ is $(k,X)$-summing it follows from \ref{step4} that $m_{2z(\omega) + Z(\omega)}(X^c)$ is not very large for $\omega$s supporting most of the mass of $A$ -- equivalently $m_{2z(\omega)+Z(\omega)}(X)$ is not too small.  Since $X$ is only slightly bigger than $A$ (this is the condition on $\eta$), on average this means that $m_{2z(\omega)+Z(\omega)}(A)$ is not too small. Now, we arranged the same uniformity properties in (\ref{step3}) for the weighted cover $(2z,Z)$ of $2\cdot S$.  It follows from this that $m_{4z(\omega)+Z(\omega)}(X)$ is not too small for a large mass of points, and this process can be iterated to get Proposition \ref{prop.modg}.
\end{enumerate}
\renewcommand{\theenumi}{\roman{enumi}}

We now turn to the details.
\begin{lemma}\label{lem.freiman}
Suppose that $A, X \subset V$; $|X| \leq K|A|$; and $A$ is $(k,X)$-summing.  Then there is a $(kK)^{-O(1)}$-hereditarily energetic set $S$ and a weighted cover $(z,Z)$ of $S$ such that
\begin{equation*}
|A\cap S| \geq (kK)^{-O(1)}|A|, |S| \leq (kK)^{O(1)}|A| \text{ and }\min_\omega{|Z(\omega)|} \geq p^{-(kK)^{O(1)}}|A| .
\end{equation*}
\end{lemma}
\begin{proof}
Apply Lemma \ref{lem.hen} to get that $A$ has $E(A) = \Omega((kK)^{-O(1)}|A|^3)$.  The Balog-Szemer{\'e}di-Gowers Theorem\footnote{In its usual form, which corresponds to \cite[Theorem 2.31((i) {$\Rightarrow$} (iv))]{taovu::} and then \cite[Exercise 2.3.15]{taovu::}).} then gives us a subset $A'\subset A$ with $|A'| = \Omega((kK)^{-O(1)}|A|)$ such that $|A'+A'| = (kK)^{O(1)}|A'|$.  By Chang's theorem for $r$-torsion groups \cite[Corollary 5.29]{taovu::} (applied to $A'$ despite only needing the large energy hypothesis) we get a subspace $U \leq V$ of size $p^{-(kK)^{O(1)}}|A'|$ such that $U \subset 2A'-2A'$.  Pl{\"u}nnecke's Inequality \cite[Corollary 6.27]{taovu::} then tells us that $|U+A'| \leq |3A'-2A'| \leq (kK)^{O(1)}|A'|$.  The claimed result follows by letting $S=U+A'$ (which has $|S+S| \leq (kK)^{O(1)}|S|$ and so is $(kK)^{-O(1)}$-hereditarily energetic by Lemma \ref{lem.hed} (\ref{pt.1})), where $z$ is the random variable taking values uniformly from $S$, and $Z$ is the constant random variable taking the value $U$.
\end{proof}
If we were interested in getting a good constant for the $O(1)$-terms in Proposition \ref{prop.modg} (and hence the $\Omega(1)$-term in Theorem \ref{thm.main}) then improvements could easily be made here.  Unusually with Freiman's theorem one is most interested in the size of the intersection of the set on the subspace, and not so much with the size of the subspace.  We give the argument we do because there are easy references to results in the literature.

We now turn to the Fourier argument in \ref{step3}.  This is a routine `energy increment' argument.
\begin{lemma}\label{lem.unif}
Suppose that $A,S \subset V$; $(z,Z)$ is a weighted cover of $S$; and $\delta \in (0,1]$ is a parameter.  Then either
\begin{enumerate}
\item
\begin{equation*}
\P(\|1_{A\cap (z+Z)} \ast (1_{A}dm_{z+Z}) -m_{z+Z}(A)^2\|_{L_2(m_{2z+Z})} \leq \delta) \geq 1-\delta;
\end{equation*}
\item or else there is a weighted cover $(z',Z')$ of $S$ with 
\begin{equation*}
\min{|Z'|} \geq p^{-O(\delta^{-2})}\min{|Z|} \text{ and } \E{m_{z'+Z'}(A)^2} \geq \E{m_{z+Z}(A)^2} + \Omega(\delta^{3}).
\end{equation*}
\end{enumerate}
\end{lemma}
\begin{proof}
Suppose that $z=x$ and $Z=U$ are such that
\begin{equation}\label{eqn.o}
\|1_{A\cap (x+U)} \ast (1_{A}dm_{x+U}) -m_{x+U}(A)^2\|_{L_2(m_{2x+U})} >\delta.
\end{equation}
Write $A':=(A-x)\cap U$ and work inside the space $U$ considered as endowed with Haar probability measure $m_U$.  That means that while the definition of $\wh{U}$ is the same as in \S\ref{sec.over}, it is convenient (for this proof) to take different normalisations for convolution and the Fourier transform: we define convolution to be
\begin{equation*}
f\ast g(y):=\int{f(x)g(y-x)dm_U(x)} \text{ for all }y\in U \text{ for all }f,g \in L_1(m_U),
\end{equation*} 
and the Fourier transform by
\begin{equation*}
\wh{f}(\gamma):=\int{f(x)\overline{\gamma(x)}dm_U(x)} \text{ for all }\gamma \in \wh{U}.
\end{equation*}
These conventions are in line with those in \cite[Definition 4.7]{taovu::} and \cite[Definition 4.6]{taovu::} respectively and the key identities are summarised in \cite[(4.1),(4.2), and (4.8)]{taovu::}.

Set
\begin{equation*}
\Gamma:=\left\{\gamma \in \wh{U}: |\wh{1_{A'}}(\gamma)| \geq \frac{1}{2}\delta \right\} \text{ and }U':=\bigcap_{\gamma \in \Gamma}{\ker \gamma}.
\end{equation*}
By Parseval's theorem for $A'$ in $U$ we have
\begin{equation*}
\left(\frac{1}{2}\delta\right)^2|\Gamma| \leq \sum_{\gamma \in \wh{U}}{|\wh{1_{A'}}(\gamma)|^2} = \int{1_{A'}(x)^2dm_U(x)} \leq 1,
\end{equation*}
so $|\Gamma| = O(\delta^{-2})$ and $|U'| \geq p^{-O(\delta^{-2})}|U|$.  On the other hand
\begin{align*}
\sum_{\gamma \in \Gamma\setminus \{0_{\wh{U}}\}}{|\wh{1_{A'}}(\gamma)|^2} + \frac{1}{2}\delta^2 & \geq \sum_{\gamma\in \Gamma\setminus\{0_{\wh{U}}\}}{|\wh{1_{A'}}(\gamma)|^4}+\left(\frac{1}{4}\delta^2\right)\sum_{\gamma\not\in \Gamma \cup \{0_{\wh{U}}\}}{|\wh{1_{A'}}(\gamma)|^2}\\ & \geq \sum_{\gamma\neq 0_{\wh{U}}}{|\wh{1_{A'}}(\gamma)|^4} = \|1_{A'} \ast 1_{A'} -m_{U}(A')^2\|_{L_2(m_{U})}^2 >
\delta^2,
\end{align*}
whence
\begin{align*}
\|1_{A\cap (x+U)} \ast m_{U'}\|_{L_2(m_{x+U})}^2& =\|1_{A'}\ast m_{U'}\|_{L_2(m_U)}^2\\ &=\sum_{\gamma \in \wh{U}}{|\wh{1_{A'}}(\gamma)|^2|\wh{m_{U'}}(\gamma)|^2}\\ & \geq \sum_{\gamma \in \Gamma}{|\wh{1_{A'}}(\gamma)|^2} > m_U(A')^2 + \frac{1}{2}\delta^2.
\end{align*}
Let $(z',Z')$ be defined conditional on $z=x$ and $Z=U$ as follows: if (\ref{eqn.o}) holds then let $Z'$ be $U'$ with (conditional) probability $1$ and choose $z'$ uniformly from $x+U$; if (\ref{eqn.o}) does not hold then let $Z'$ be $U$ with (conditional) probability $1$ and $z'=x$ with (conditional) probability $1$.  It then follows that $(z',Z')$ is a weighted cover of $S$ and
\begin{equation*}
\E{m_{z'+Z'}(A)^2} \geq  \E{m_{z+Z}(A)^2} + \P(\|1_{A\cap (z+Z)} \ast (1_{A}dm_{z+Z}) -m_{z+Z}(A)^2\|_{L_2(m_{2z+Z})} > \delta)\cdot \frac{1}{2}\delta^2.
\end{equation*}
It follows that if we are not in the first case of the lemma then we must be in the second.
\end{proof}
If we were interested in optimising our arguments then it might be more effective to use an $L_p$-version of the above in the style of Croot and Sisask (compare, for example, \cite[Proposition 3.3]{crosis::} with \cite[Proposition 3.1]{crosis::}).

The above lemma leads to a weighted cover $(z,Z)$ where $Z$ is not necessarily constant.  This is necessary for us to ensure good bounds -- if we wanted $Z$ to be constant then examples such as those in \cite[Theorem 10.2]{gre::02} show we would have to have tower-type dependencies.  This phenomenon is discussed after \cite[Proposition 5.8]{gre::9} as an important part of Shkredov's argument from \cite{shk::1} (see also \cite{shk::00}).
\begin{corollary}\label{cor.y}
Suppose that $A,X \subset V$; $|X| \leq K|A|$; $A$ is $(k,X)$-summing; and $r \in \N$ and $\delta \in (0,1]$ are parameters.  Then there is a $(kK)^{-O(1)}$-hereditarily energetic set $S$ and a weighted cover $(z,Z)$ such that
\begin{equation*}
|A\cap S| \geq (kK)^{-O(1)}|A|, |S| \leq (kK)^{O(1)}|A| \text{ and }\min_\omega{|Z(\omega)|} \geq p^{-(\delta^{-1}rkK)^{O(1)}}|A|,
\end{equation*}
and for all $0\leq i \leq r-1$ we have
\begin{equation*}
\P(\|1_{A\cap (2^iz+Z)} \ast (1_{A}dm_{2^iz+Z}) -m_{2^iz+Z}(A)^2\|_{L_2(m_{2^{i+1}z+Z})} \leq \delta) \geq 1-\delta.
\end{equation*}
\end{corollary}
\begin{proof}
We first apply Lemma \ref{lem.freiman} to get a $(kK)^{-O(1)}$-hereditarily energetic set $S$ and a weighted cover $(z_0,Z_0)$ of $S$ such that
\begin{equation*}
|A\cap S| \geq (kK)^{-O(1)}|A|, |S| \leq (kK)^{O(1)}|A| \text{ and }\min_\omega\{|Z_0(\omega)|\} \geq p^{-(kK)^{O(1)}}|A| .
\end{equation*}
Suppose that we have a weighted cover $(z_i,Z_i)$ of $S$ and there is some $0 \leq j < r$ such that
\begin{equation*}
\P(\|1_{A\cap (2^jz_i+Z_i)} \ast (1_{A}dm_{2^jz_i+Z_i}) -m_{2^jz_i+Z_i}(A)^2\|_{L_2(m_{2^{j+1}z_i+Z_i})} \leq \delta) < 1-\delta.
\end{equation*}
Then apply Lemma \ref{lem.unif} to the weighted cover $(2^jz_i,Z_i)$ of $2^j \cdot S$.  We get a new weighted cover $(z',Z_{i+1})$ of $2^j\cdot S$ (and we put $z_{i+1}=2^{-j}z'$ so $(z_{i+1},Z_{i+1})$ is a weighted cover of $S$) such that
\begin{equation*}
\min_\omega{|Z_{i+1}(\omega)|} \geq p^{-O(\delta^{-2})}\min_\omega{|Z_i(\omega)|} \text{ and } \E{m_{2^jz_{i+1}+Z_{i+1}}(A)^2} \geq \E{m_{2^jz_i+Z_i}(A)^2} + \Omega(\delta^3).
\end{equation*}
Since for any weighted cover $(z,Z)$ of $S$ we have
\begin{equation*}
\sum_{j=0}^{r-1}{\E{m_{2^jz+Z}(A)^2}} \leq r
\end{equation*}
we see that we can be in this situation at most $O(\delta^{-3}r)$ times.  The iteration terminates with the desired outcome.
\end{proof}
The second key ingredient as far as bounds are concerned comes in the next lemma covering \ref{step4}.  The level of uniformity necessary to count in the case we are interested in is polynomial in $k$ whereas in \emph{e.g.} \cite[Corollary 4.2]{sha::0} it is exponential in $k$.
\begin{lemma}\label{lem.ct}
Suppose that $x \in V$; $U \leq V$; $A,X \subset V$; $\alpha:=m_{x+U}(A)$ and $m_{2x+U}(X) \leq \epsilon$;
\begin{equation*}
\|1_{A\cap (x+U)} \ast (1_{A}dm_{x+U}) -\alpha^2\|_{L_2(m_{2x+U})} \leq \epsilon\alpha^2;
\end{equation*}
and $k\in \N$ is a parameter.  Then either $k=\Omega(\epsilon^{-\frac{1}{2}})$ or
\begin{equation}\label{eqn.qqqqq}
\int{\left(\prod_{1 \leq i <j \leq k}{1_{(2x+U) \setminus X}(z_i+z_j)}\right)\prod_{i=1}^k{1_A(z_i)dm_{x+U}(z_i)}} = \Omega(\alpha^k).
\end{equation}
\end{lemma}
\begin{proof}
First note that
\begin{align*}
& \int{1_X(z_1+z_2)1_A(z_1)1_A(z_2)dm_{x+U}(z_1)dm_{x+U}(z_2)}\\
& \qquad \qquad  = \alpha^2- \langle 1_{A\cap (x+U)} \ast (1_{A}dm_{x+U}) ,1_{(2x+U) \setminus X}\rangle_{L_2(m_{2x+U})}\\ & \qquad \qquad \leq \alpha^2- \alpha^2m_{2\cdot {x+U}}((2x+U) \setminus X) + \epsilon \alpha^2 \leq 2\epsilon\alpha^2.
\end{align*}
Now write $Q$ for the integral on the left of (\ref{eqn.qqqqq}). Apply Bonferroni's inequality\footnote{In the form $\P(\bigcap_i{E_i}) \geq 1- \sum_i{\P(E_i^c)}$.} to the events $\{z \in (A\cap (x+U))^k: z_i+z_j \in X\}$
\begin{align*}
Q & \geq \int{ \left(\prod_{1 \leq i <j \leq k}{1_{2x+U}(z_i+z_j)}\right)\prod_{i=1}^k{1_A(z_i)dm_{x+U}(z_i)}}\\
& \qquad - \sum_{1\leq i'<j'\leq k}{\int{1_{X}(z_{i'}+z_{j'})\prod_{\substack{1 \leq i < j \leq k\\(i,j)\neq (i',j')}}{1_{2\cdot {x+U}}(z_i+z_j)}\prod_{i=1}^k{1_A(z_i)dm_{x+U}(z_i)}}}\\
& = \alpha^k -\sum_{1 \leq i'<j' \leq k}{\int{1_{X}(z_{i'}+z_{j'})\prod_{i=1}^k{1_A(z_i)dm_{x+U}(z_i)}}} \geq \alpha^k - 2\epsilon \binom{k}{2}\alpha^k.
\end{align*}
The lemma follows.
\end{proof}
It seems likely that a more careful argument using Tur{\'a}n's theorem (in a form like \cite[Lemma 3.1]{sudszevu::}) or the Lov{\'a}sz Local Lemma \cite[Corollary 1.2.6]{taovu::} could be used to let us take $k=\Omega(\epsilon^{-1})$.  Again, such improvements would impact the $O(1)$-term in Proposition \ref{prop.modg} and, ultimately, the $\Omega(1)$-term in Theorem \ref{thm.main}, but this is not the concern of the present paper.

For us the crucial aspect of the above is that the conclusion $k=\Omega(\epsilon^{-\frac{1}{2}})$ does not depend on the density $\alpha$.  If it were allowed to depend on $\alpha$ then we would not need the uniformity argument in Lemma \ref{lem.unif}.  The reason that it is not is that it would lead to a lower bound on the intersections $(2^j\cdot S) \cap A$ which decreases with $j$.  This in turn is not enough for our application in the non-model setting.

Finally we have the tools to complete the argument -- \ref{step5}.
\begin{proof}[Proof of Proposition \ref{prop.modg}]
Take $\epsilon_0 = \Omega(k^{-2})$ such that the first conclusion of Lemma \ref{lem.ct} does not happen and $\nu_0=k^{-O(1)}$ be such that $m_S(A) \geq \nu_0$ always holds in the conclusion of Corollary \ref{cor.y} when $K\leq 2$.

Apply Corollary \ref{cor.y} with $K=2$ and $\delta=2^{-3}r^{-1}\epsilon_0\min\{ \nu_0^2,\epsilon_0^2\}$ to get a $k^{-O(1)}$-hereditarily energetic set $S$ and a weighted cover $(z,Z)$ (supported on the probability space $(\Omega,\P)$) such that
\begin{equation}\label{eqn.jut}
|A\cap S| \geq k^{-O(1)}|A|, |S| \leq k^{O(1)}|A| \text{ and }\min_\omega{|Z(\omega)|} \geq p^{-(rk)^{O(1)}}|A|,
\end{equation}
and, for $0 \leq s <r$ writing
\begin{equation*}
E_s:=\left\{\omega \in \Omega : \|1_{A\cap (2^sz+Z)} \ast (1_{A}dm_{2^sz+Z}) -m_{2^sz+Z}(A)^2\|_{L_2\left(m_{2^{s+1}z+Z}\right)} \leq \delta\right\}
\end{equation*}
we have $\P(E_s^c)<\delta$.  For $1\leq s \leq r$ write
\begin{equation*}
L_s:=\left\{\omega \in \Omega :m_{2^sz+Z}(X\setminus A)\leq \frac{1}{2}\epsilon_0\right\}.
\end{equation*}
Then
\begin{equation*}
\P(L_s^c) < 2\epsilon_0^{-1}\E{m_{2^sz+Z}(X\setminus A)} =2\epsilon_0^{-1}m_{2^s\cdot S}(X \setminus A) \leq k^{O(1)}\frac{|X \setminus A|}{|A|} = \eta k^{O(1)}.
\end{equation*}
Either $\eta^{-1}=(rk)^{O(1)}$ or else $\P(L_s^c) \leq \frac{1}{8r}\nu_0$ for all $1 \leq s \leq r$; we may assume the latter.  Write
\begin{equation*}
B:=\left\{\omega \in \Omega: m_{z+Z}(A)>\frac{1}{2}\nu_0\right\} \text{ and } \Omega':=B\cap \left(\bigcap_{i=1}^r{L_i}\right)\cap \left(\bigcap_{i=0}^{r-1}{E_i}\right),
\end{equation*}
and note that
\begin{align*}
\E{m_{z+Z}(A)1_{\Omega'}}& \geq \E{m_{z+Z}(A)} -\E{m_{z+Z}(A)\left(1_{B^c} +\sum_{s=1}^r{1_{L_s^c}} + \sum_{s=0}^{r-1}{1_{E_s^c}}\right)}\\ & \geq \nu_0 -\frac{1}{2}\nu_0- r\cdot \frac{1}{8r}\nu_0 - r\delta \geq \frac{1}{8}\nu_0.
\end{align*}

\begin{claim*}
Either $|A| \leq p^{(kr)^{O(1)}}$ or else for all $\omega \in \Omega'$ we have
\begin{equation*}
m_{2^sz(\omega)+Z(\omega)}(A)>\min\left\{\frac{1}{2}\nu_0,\frac{1}{2}\epsilon_0\right\} \text{ for all }0 \leq s \leq r.
\end{equation*}
\end{claim*}
\begin{proof}
We proceed by induction.  When $s=0$ we have $\omega \in B$ and so the hypothesis holds.  Suppose that it holds for $0 \leq s <r$.  Then since $\omega \in E_s$
\begin{align*}
&\|1_{A\cap (2^sz(\omega)+Z(\omega))} \ast (1_{A}dm_{2^sz(\omega)+Z(\omega)}) -m_{2^sz(\omega)+Z(\omega)}(A)^2\|_{L_2\left(m_{2^{s+1}z(\omega)+Z(\omega)}\right)}\\ & \qquad \qquad \qquad \qquad \qquad \qquad \qquad \qquad \qquad \qquad \qquad \qquad\leq \epsilon_0  m_{2^sz(\omega)+Z(\omega)}(A)^2.
\end{align*}
If $m_{2^{s+1}z(\omega)+Z(\omega)}(X) \leq \epsilon_0$ then it follows from Lemma \ref{lem.ct} that
\begin{align*}
& \int{\left(\prod_{1 \leq i <j \leq k}{1_{(2^{s+1}z(\omega)+Z(\omega)) \setminus X}(z_i+z_j)}\right)\prod_{i=1}^k{1_A(z_i)dm_{2^sz(\omega)+Z(\omega)}(z_i)}}\\ & \qquad \qquad \qquad \qquad \qquad \qquad \qquad \qquad \qquad \qquad  =\Omega\left(\left(\min\left\{\frac{1}{2}\nu_0,\frac{1}{2}\epsilon_0\right\}\right)^k\right).
\end{align*}
Since $A$ is $(k,X)$-summing we know that the first product on the left is $0$ on $z \in A^k$ unless there is $1 \leq i < j \leq k$ with $z_i=z_j$.  It follows that the integral is at most
\begin{equation*}
\binom{k}{2}m_{2^sz(\omega)+Z(\omega)}(A)^{k-1}\cdot \frac{1}{|2^sz(\omega)+Z(\omega)|}.
\end{equation*}
The upper bound on $|A|$ follows from the lower bound on $|Z(\omega)|$ in (\ref{eqn.jut}).  Thus we may assume $m_{2^{s+1}z(\omega)+Z(\omega)}(X) >\epsilon_0$, and since $\omega \in L_{s+1}$ we have
\begin{equation*}
m_{2^{s+1}z(\omega)+Z(\omega)}(A) \geq m_{2^{s+1}z(\omega)+Z(\omega)}(X) - m_{2^{s+1}z(\omega)+Z(\omega)}(X \setminus A) > \frac{1}{2}\epsilon_0 \geq \min\left\{\frac{1}{2}\nu_0,\frac{1}{2}\epsilon_0\right\}.
\end{equation*}
The claim is proved.
\end{proof}

We conclude that for all $1 \leq s \leq r$ we have
\begin{align*}
m_{2^s\cdot S}(A) = \E{m_{2^sz+Z}(A)} \geq \E{m_{2^sz+Z}(A)1_{\Omega'}} & \geq \min\left\{\frac{1}{2}\nu_0,\frac{1}{2}\epsilon_0\right\}\E{1_{\Omega'}}\\ & \geq  \min\left\{\frac{1}{2}\nu_0,\frac{1}{2}\epsilon_0\right\}\E{m_{z+Z}(A)1_{\Omega'}}= k^{-O(1)}
\end{align*}
as required.
\end{proof}

\section{Translating the model}

The remainder of the paper deals with converting the model argument of \S\ref{sec.mod} to give a proof of Proposition \ref{prop.g}; throughout $G$ denotes an abelian group.  This is analogous to converting Meshulam's proof of the Roth-Meshulam Theorem \cite[Proposition 10.12]{taovu::} to Bourgain's proof of Roth's Theorem \cite[Theorem 10.29]{taovu::}.

The reader already familiar with this sort of translation can move to \S\ref{sec.pf}.  The key definitions are $\tau$-closed pairs, defined before Lemma \ref{lem.basicclosed}; covering numbers -- $\mathcal{C}$ and $\mathcal{C}^\flat$, defined in (\ref{eqn.coveringnumber}) and before Lemma \ref{lem.cnum} respectively; and $\mathcal{S}(G)$, systems, and dimension defined after the proof of Lemma \ref{lem.cnum}.  The main variation in the definitions we have chosen here is in defining $\mathcal{C}^\flat$, which is set up in the way it is so that dimension is sub-additive with respect to intersection, rather than sub-additive up to a multiplicative constant.

The basic idea is to replace subspaces by pairs of sets which are `almost' closed, and we start with a definition for this purpose.  We write $\mathcal{N}(G)$ for the set of finite symmetric neighbourhoods of the identity, that is sets $S \subset G$ with $S=-S$ and $0_G \in S$.  (We use topological language here for motivation -- the systems we define later can be thought of as bases for particular topologies.) Given $Z,W \in \mathcal{N}(G)$ we say that $(Z,W)$ is \textbf{$\tau$-closed} if there are sets $Z^-,Z^+ \in \mathcal{N}(G)$ such that
\begin{equation*}
Z^-+W \subset Z,Z+W \subset Z^+ \text{ and } |Z^+| \leq (1+\tau)|Z^-|.
\end{equation*}
Given a measure $\mu$ on $G$ and $x \in G$ we write $\tau_x(\mu)$ for the measure induced by
\begin{equation*}
C(G) \rightarrow C(G); f \mapsto \left(y \mapsto \int{f(y-x)d\mu(y)}\right).
\end{equation*}
As in \S\ref{sec.mod} if $S$ is a finite non-empty subset of $G$ we write $m_S$ for the uniform probability measure supported on $S$.
\begin{lemma}[Basic properties of $\tau$-closed pairs]\label{lem.basicclosed} Suppose that $(Z,W)$ is a $\tau$-closed.  Then
\begin{enumerate}
\item \label{pt.inv} $\|\tau_w(m_Z)-m_Z\| \leq \tau$ for all $w \in W$;
\item \label{pt.nesting} if $W' \subset W$ then $(Z,W')$ is $\tau$-closed;
\item \label{pt.dilation} and if $G$ has no $2$-torsion then $(2^m\cdot Z,2^m\cdot W)$ is $\tau$-closed.
\end{enumerate}
\end{lemma}
\begin{proof}
The first property is immediate by the triangle inequality; the second immediate; and the third equally so once we recall that $x\mapsto 2^mx$ is an injection in a group with no $2$-torsion.
\end{proof}
These sorts of pairs behave enough like subspaces that we can prove counting-type results (as we shall in \S\ref{sec.k}), but not so much that they are amenable to iteration without greater losses than we can afford to bear.  One usually deals with this by recording auxiliary data in the form of (the-data-necessary-to-generate) Bohr sets.  We discuss this in detail before Lemma \ref{lem.bohr}.

The fact we are using Freiman's theorem suggests that we actually need Bohr sets inside coset progressions (defined before Lemma \ref{lem.gap}), and the notion of a Bourgain system was formulated in \cite[Definition 4.1]{gresan::0} to deal with exactly this situation.  This gives a common framework for Bohr sets and generalised arithmetic progressions and it turns out much of the technology Bourgain developed for Bohr sets in \cite{bou::5} extends.  We shall use a very similar definition.

It turns out that we only need to use Freiman's theorem once so rather than proceeding with the more general systems below we could simply pass to a long arithmetic progression within the generalised arithmetic progression (in the style of \cite[Exercise 3.2.5]{taovu::}) and consider Bohr sets inside that.  This does not seem to afford any great simplification.

We now turn to the basic definitions.  For $X,Y \subset G$ we write
\begin{equation}\label{eqn.coveringnumber}
\mathcal{C}(X;Y):=\min\{|T|: X \subset T+Y\}.
\end{equation}
We call these numbers covering numbers.  Here and throughout we take the usual conventions concerning $\infty$.
\begin{lemma}[Basic properties of covering numbers]\label{lem.bcn}Suppose that $G$ and $H$ are abelian groups.
\begin{enumerate}
\item\label{cpt.chain} \emph{(Chaining)} For all $X,Y,Z \subset G$ we have
\begin{equation*}
\mathcal{C}(X;Z) \leq \mathcal{C}(X;Y)\mathcal{C}(Y;Z).
\end{equation*}
\item\label{cpt.hom} \emph{(Homomorphisms)} For all $X,Y \subset G$ and homomorphisms $\phi:G \rightarrow H$ we have
\begin{equation*}
\mathcal{C}(\phi(X);\phi(Y)) \leq \mathcal{C}(X;Y).
\end{equation*}
\item\label{cpt.cs} \emph{(Covering and size)} For all $X,Y\subset G$ we have
\begin{equation*}
|X| \leq \mathcal{C}(X;Y)|Y|.
\end{equation*}
\item\label{cpt.rcl} \emph{(Ruzsa's covering lemma)} For all $X,Y\subset G$ we have
\begin{equation*}
\mathcal{C}(X;Y-Y) \leq \min\left\{\frac{|X+Y|}{|Y|},\frac{|X-Y|}{|Y|}\right\}.
\end{equation*}
\end{enumerate}
\end{lemma}
\begin{proof}
The first three facts are immediate from the definition.  The last is just Ruzsa's covering lemma \cite[Lemma 2.14]{taovu::}, which can be proved by letting $T\subset X$ be a maximal subset such that if $(t+Y) \cap (t'+Y)\neq \emptyset $ and $t,t' \in T$ then $t=t'$.  This gives the first bound in the minimum; applying the first with $Y$ replaced by $-Y$ and noticing that $(-Y)-(-Y) = Y-Y$ and $|-Y|=|Y|$ gives the second.
\end{proof}

Covering numbers do not behave well with respect to intersections and so we define\footnote{Here $\textbf{Ab}$ denotes the category of abelian groups.  One could equally say that $H$ is an abelian group and $\phi$ is a homomorphism from $G$ to $H$.}
\begin{equation*}
\mathcal{C}^\flat(X;Y):=\min_{\substack{H\in \textbf{Ab}\\ \phi\in \Hom(\langle X\rangle,H) }}{\left\{\mathcal{C}(W;Z): Z,W\subset H, X \subset \phi^{-1}(W) \text{ and } \phi^{-1}(Z-Z) \subset Y\right\}}.
\end{equation*}
Here $\langle X\rangle$ denotes the group generated by $X$ and is there to make the definition independent of the particular ambient group in which $X$ and $Y$ live.

These numbers are close to covering numbers but also respect intersections.
\begin{lemma}[Basic properties of $\mathcal{C}^\flat$]\label{lem.cnum}Suppose that $G$ and $K$ are abelian groups.
\begin{enumerate}
\item\label{pt.order} \emph{(Order)} For all $X' \subset X$ and $Y \subset Y'$ we have
\begin{equation*}
\mathcal{C}^\flat(X';Y') \leq \mathcal{C}^\flat(X;Y).
\end{equation*}
\item\label{pt.hom} \emph{(Homomorphisms)} For all $X,Y \subset G$ and $\psi:K \rightarrow G$ a homomorphism we have
\begin{equation*}
\mathcal{C}^\flat(\psi^{-1}(X);\psi^{-1}(Y)) \leq \mathcal{C}^\flat(X;Y).
\end{equation*}
\item\label{pt.equiv} \emph{(Equivalence)} Whenever $X,Y\subset G$ we have
\begin{equation*}
\mathcal{C}^\flat(X;Y-Y) \leq \mathcal{C}(X;Y) \leq \mathcal{C}^\flat(X;Y).
\end{equation*}
\item\label{pt.meet} \emph{(Meets)} Whenever $X,X',Y,Y' \subset G$ we have
\begin{equation*}
\mathcal{C}^\flat(X\cap X';Y\cap Y') \leq \mathcal{C}^\flat(X;Y)\mathcal{C}^\flat(X';Y').
\end{equation*}
\end{enumerate}
\end{lemma}
\begin{proof}
Suppose for all parts that $H \in \textbf{Ab}$, $\phi \in \Hom(\langle X\rangle,H)$, and $Z,W \subset H$ are such that $X \subset \phi^{-1}(W)$ and $\phi^{-1}(Z-Z) \subset Y$, and $\mathcal{C}^\flat(X;Y)=\mathcal{C}(W;Z)$.

For (\ref{pt.order}) note $\psi:=\phi|_{\langle X'\rangle} \in \Hom(\langle X'\rangle,H)$ and $X' \subset X \cap \langle X'\rangle \subset \psi^{-1}(W)$ and $\psi^{-1}(Z-Z) \subset Y\subset Y'$.  It follows that $\mathcal{C}^\flat(X';Y') \leq \mathcal{C}(W;Z)=\mathcal{C}^\flat(X;Y)$ as claimed.

For (\ref{pt.hom}) note $\pi:=(\phi\circ \psi)|_{\langle \psi^{-1}(X)\rangle} \in \Hom(\langle \psi^{-1}(X)\rangle,H)$ and $\psi^{-1}(X) \subset \pi^{-1}(W)$ and $\pi^{-1}(Z-Z)\subset \psi^{-1}(Y)$.  Again, it follows that $\mathcal{C}^\flat(X';Y') \leq \mathcal{C}(W;Z)=\mathcal{C}^\flat(X;Y)$ as claimed.

For (\ref{pt.equiv}) we get the left hand inequality by considering the canonical embedding $\phi:\langle X\rangle \rightarrow G \in \Hom(\langle X\rangle,G)$, and noting that $\phi^{-1}(Y-Y) \subset Y-Y$ and $X \subset \phi^{-1}(X)$.

For the right hand inequality let $S \subset H$ be such that $W \subset S+Z$ and $|S| = \mathcal{C}(W;Z)$, and $T \subset \langle X\rangle$ be minimal such that if $s \in S$ has $(s+Z) \cap \phi(\langle X\rangle)\neq \emptyset$ then there is some $t \in T$ such that $\phi(t) \in s+Z$.  It follows that $|T| \leq |S|$, and if $x \in X$ then $x \in \phi^{-1}(W)$ and so $\phi(x) \in W$ and there is some $s \in S$ with $\phi(x) \in s+Z$.  By the definition of $T$ there is some $t \in T$ such that $\phi(t) \in s+Z$, whence $\phi(x-t) \in Z-Z$ and so $x \in T+\phi^{-1}(Z-Z) \subset T+Y$.  We conclude that $\mathcal{C}(X;Y) \leq |T| \leq |S|$.  The claimed inequality follows.

Finally, for (\ref{pt.meet}), suppose additionally that $H' \in \textbf{Ab}$, $\psi\in \Hom(\langle X'\rangle,H')$, and $Z',W' \subset H'$ are such that $X' \subset \psi^{-1}(W')$ and $\psi^{-1}(Z'-Z') \subset Y'$, and $\mathcal{C}^\flat(X';Y')=\mathcal{C}(W';Z')$.

Then $H \times H'$ is an abelian group; $W\times W', Z \times Z' \subset H\times H'$; $\pi:\langle X\cap X'\rangle \rightarrow H \times H'; x \mapsto (\phi(x),\psi(x))$ is a homomorphism; and
\begin{equation*}
\pi^{-1}(W \times W') = \langle X\cap X'\rangle \cap(\phi^{-1}(W) \cap \psi^{-1}(W')) \supset X \cap X',
\end{equation*}
and
\begin{align*}
\pi^{-1}((Z \times Z') - (Z \times Z')) & = \pi^{-1}((Z -Z)\times (Z'-Z'))\\ & \subset \phi^{-1}(Z-Z)\cap \psi^{-1}(Z'-Z') \subset Y\cap Y'.
\end{align*}
On the other hand,
\begin{equation*}
\mathcal{C}(W \times W';Z\times Z') \leq \mathcal{C}(W;Z) \mathcal{C}(W'; Z')=\mathcal{C}^\flat(X;Y)\mathcal{C}^\flat(X';Y').
\end{equation*}
The result follows.
\end{proof}
We use a slight variant of \cite[Definition 4.1]{gresan::0}, defining a \textbf{system} on $G$ to be a vector $B=(B_i)_{i \in \N_0}$ of symmetric neighbourhoods of the identity such that $B_{i+1}+B_{i+1} \subset B_i$ for all $i \in \N_0$.  We define its \textbf{dimension} to be
\begin{equation*}
\dim B:=\sup_{i \in \N_0}{\log_2\mathcal{C}^\flat(B_i;B_{i+1})},
\end{equation*}
and write $\mathcal{S}(G)$ for the set of systems on $G$.  For $S \subset G$ we shall also write $\mathcal{C}^\flat(S;B)$ for $\mathcal{C}^\flat(S;B_0)$.  This is how we record the `size' of $B$ relative to some reference set $S$.  (As an aside we remark that while we have found it convenient to use this notion of `size' the more conventional $|B_0|$ works better in some ways: while it does not deal so well with intersections, it would allow us to dispense with the second part of Lemma \ref{lem.new} (\ref{pt.mult}) below.)

There are many examples of systems: coset progressions naturally define systems as we shall show in Lemma \ref{lem.gap}, as do Bohr sets (which we cover in Lemma \ref{lem.bohr}), and subgroups which we deal with now.
\begin{lemma}[Subgroup systems]\label{lem.subgroup}
Suppose that $H \leq G$.  Then the $\N_0$-indexed vector $B$ taking the constant value $H$ is a $0$-dimensional system.
\end{lemma}
\begin{proof}
$B$ is easily seen to be a system.  Moreover $\mathcal{C}(H;H)\leq \mathcal{C}^\flat(H;H) = \mathcal{C}^\flat(H;H-H) \leq \mathcal{C}(H;H)$ by Lemma \ref{lem.cnum} (\ref{pt.equiv}) and $\mathcal{C}(H;H)=1$, so $\dim B=0$ as claimed.
\end{proof}

There are three ways of creating new systems that will be useful to us. Given $B,B' \in \mathcal{S}(G)$ and $m \in\N_0$ we make the following definitions.
\begin{itemize}
\item The \textbf{intersection} of $B$ and $B'$ is the vector $B\wedge B':=(B_i \cap B_i')_{i \in \N_0}$.  It is easy to check that $(\mathcal{S}(G),\wedge)$ forms a meet semi-lattice.  (Meaning that $B\wedge B'$ is indeed a system; that $B\wedge (B'\wedge B'') = (B\wedge B') \wedge B''$; that $B \wedge B' = B'\wedge B$; and that $B\wedge B=B$.)
\item The \textbf{$2^{-m}$-dilate} of $B$ is the vector $2^{-m}B:=(B_{i+m})_{i \in \N_0}$.  It is easy to check that this is an action of the (additive) monoid $\N_0$ on $\mathcal{S}(G)$.  (Meaning that $2^{-m}B$ is indeed a system; that $1B=B$; and $2^{-m}(2^{-m'}B) = 2^{-(m+m')}B$.)
\item The \textbf{$2^m$-multiple} of $B$ is the vector $2^m\cdot B :=(2^m\cdot B_i)_{i \in \N_0}$.  Again, it is easy to check that this is an action of $\N_0$ on $\mathcal{S}(G)$.
\end{itemize}
The meet semi-lattice structure induces a partial order on $\mathcal{S}(G)$, and $B' \leq B$ if and only if $B'_i \subset B_i$ for all $i\in \N_0$.

Dilates distribute over intersection, meaning that $2^{-m}(B\wedge B') = (2^{-m}B) \wedge (2^{-m}B')$, but in general for multiples we only have $2^m \cdot (B \wedge B') \leq (2^m\cdot B) \wedge (2^m \cdot B')$; if $G$ has no $2$-torsion then we do have equality.

Finally it is worth noting that multiples and dilates do not interact terribly well: in particular we will need to consider systems of the form $2^m\cdot (2^{-m'}B)$ and this does \emph{not} in general simplify.
\begin{lemma}\label{lem.new} Suppose that $B,B' \in \mathcal{S}(G)$; $S \subset G$; and $m \in \N_0$.  Then
\begin{enumerate}
\item\label{pt.intdim} \emph{(Intersections)}
\begin{equation*}
\dim B\wedge B' \leq \dim B + \dim B' \text{ and }\mathcal{C}^\flat(S;B\wedge B') \leq \mathcal{C}^\flat(S;B)\mathcal{C}^\flat(S;B');
\end{equation*}
\item\label{pt.szy} \emph{(Dilations)}
\begin{equation*}
\dim 2^{-m}B \leq \dim B \text{ and }\mathcal{C}^\flat(S;2^{-m}B) \leq \mathcal{C}^\flat(S;B)2^{(m+1)\dim B};
\end{equation*}
\item \label{pt.mult}\emph{(Multiples)}
\begin{equation*}
\dim 2^{m}\cdot B \leq \dim B,
\end{equation*}
and if $G$ has no $2$-torsion then
\begin{equation*}
\mathcal{C}^\flat(S;2^{m}\cdot B) \leq \mathcal{C}^\flat(S;B)\exp(O(m\dim B)).
\end{equation*}
\end{enumerate}
\end{lemma}
\begin{proof}
The first part follows immediately from Lemma \ref{lem.cnum} (\ref{pt.meet}) and the definitions.

For (\ref{pt.szy}) the dimension inequality is immediate from the definition.  For the second estimate we have the following chain of inequalities justified afterwards.
\begin{align*}
\mathcal{C}^\flat(S;2^{-m}B) & = \mathcal{C}^\flat(S;B_m)\\ & \leq \mathcal{C}^\flat(X;B_{m+1}-B_{m+1})\\ & \leq \mathcal{C}(X;B_{m+1}) \leq \mathcal{C}(X;B_0)\prod_{i=0}^{m}{\mathcal{C}(B_i;B_{i+1})} \leq \mathcal{C}^\flat(X;B_0)2^{(m+1)\dim B}.
\end{align*}
$B$ is a system so $B_{m+1}-B_{m+1} \subset B_m$, and so Lemma \ref{lem.cnum} (\ref{pt.order}) gives the first inequality.  The second follows from the first inequality in Lemma \ref{lem.cnum} (\ref{pt.equiv}).  The third by Lemma \ref{lem.bcn} (\ref{cpt.chain}), and then the fourth by the second inequality in Lemma \ref{lem.cnum} (\ref{pt.equiv}).

For the dimension bound in (\ref{pt.mult}) note that the isomorphism $\psi:2^m \cdot G \rightarrow G; x \mapsto 2^{-m}x$ has $\psi^{-1}(B_i)=(2^m\cdot B)_i$ and the bound follows from Lemma \ref{lem.cnum} (\ref{pt.hom}) and the definition of dimension.

For the second part of (\ref{pt.mult}) suppose that $G$ has no $2$-torsion.  We have the same chain of inequalities as above.  Again, they are justified afterwards.
\begin{align*}
\mathcal{C}^\flat(S;2^m\cdot B) & =\mathcal{C}^\flat(S;2^m\cdot B_0)\\&\leq \mathcal{C}^\flat(S;2^m\cdot B_1 - 2^m \cdot B_1)\\ & \leq \mathcal{C}(S;2^m\cdot B_1) \leq \mathcal{C}(S;B_0)\mathcal{C}(B_0;B_1)\prod_{t=0}^{m-1}{\mathcal{C}(2^t\cdot B_1; 2^{t+1}\cdot B_1)}.
\end{align*}
$2^m\cdot B$ is a system so $2^m\cdot B_1 - 2^m \cdot B_1 \subset 2^m \cdot B_0$, and so Lemma \ref{lem.cnum} (\ref{pt.order}) gives the first inequality. The second follows from the first inequality in Lemma \ref{lem.cnum} (\ref{pt.equiv}), and then the third by Lemma \ref{lem.bcn} (\ref{cpt.chain}).

By Lemma \ref{lem.bcn} (\ref{cpt.hom}), the fact that $2\cdot B$ is a system so $2\cdot B_2 - 2\cdot B_2 \subset 2\cdot B_1$, and Ruzsa's covering lemma (Lemma \ref{lem.bcn} (\ref{cpt.rcl})) we have
\begin{equation*}
\mathcal{C}(2^t\cdot B_1; 2^{t+1}\cdot B_1) \leq  \mathcal{C}(B_1;2\cdot B_1)\leq \mathcal{C}(B_1;2\cdot B_2-2\cdot B_2) \leq \frac{|B_1+2\cdot B_2|}{|2\cdot B_2|}.
\end{equation*}
Since
\begin{equation*}
B_1+2\cdot B_2 \subset B_1+B_2+B_2\subset B_1+B_1 \subset B_0,
\end{equation*}
we have by Lemma \ref{lem.bcn} (\ref{cpt.cs}) and (\ref{cpt.chain}) that
\begin{align*}
|B_1+2\cdot B_2| &\leq |B_0| \leq \mathcal{C}(B_0;B_1)\mathcal{C}(B_1;B_2)|B_2|.
\end{align*}
Combining all this gives
\begin{equation*}
\mathcal{C}^\flat(S;2^m\cdot B) \leq \mathcal{C}(S;B_0)\mathcal{C}(B_0;B_1)\prod_{t=0}^{m-1}{\left( \mathcal{C}(B_0;B_1)\mathcal{C}(B_1;B_2)\frac{|B_2|}{|2\cdot B_2|}\right)}.
\end{equation*}
Since $G$ has no $2$-torsion $|2\cdot B_2|=|B_2|$ and so the second part of Lemma \ref{lem.cnum} (\ref{pt.equiv}) can then be used to give the result.
\end{proof}
The requirement that $G$ has no $2$-torsion is clearly necessary for the second part of (\ref{pt.mult}) above; the proof is essentially the same as the proof of \cite[Theorem 15]{buk::0} with Ruzsa's triangle inequality \cite[Lemma 2.6]{taovu::} replaced by Lemma \ref{lem.bcn} (\ref{cpt.chain}).

At the start of the section we introduced the notion of $\tau$-closed pair and we can use the pigeonhole principle to produce a plentiful supply of such pairs from systems.  Although we do not need it, a stronger result \cite[Lemma 4.24]{taovu::} is available.
\begin{lemma}\label{lem.ubreg}
Suppose that $Z \in \mathcal{N}(G)$; $B \in \mathcal{S}(G)$; $|Z+B_0| \leq K|Z|$; $\tau \in (0,1]$ is a parameter.  Then there is a set $Z_0 \in \mathcal{N}(G)$ with $Z\subset Z_0 \subset Z+B_0$ and a natural $m=\log_2\log 2 K+\log_2\tau^{-1} + O(1)$ such that $(Z_0,B_m)$ is $\tau$-closed.  
\end{lemma}
\begin{proof}
Let $m \in \N_0$ be a parameter to be optimised shortly
\begin{equation*}
\prod_{j=0}^{2^{m-1}-1}{\frac{|Z+(2j+2)B_m|}{|Z+2jB_m|}} \leq \frac{|Z+2^mB_m|}{|Z|} \leq \frac{|Z+B_0|}{|Z|}\leq K.
\end{equation*}
By the pigeonhole principle there is some $0 \leq j \leq 2^{m-1}-1$ such that
\begin{equation*}
\frac{|Z+(2j+2)B_m|}{|Z+2jB_m|} \leq K^{\frac{1}{2^{m-1}}},
\end{equation*}
so we can take $m=\log_2\log_2K + \log_2\tau^{-1} + O(1)$ such that the right hand side is at most $1+\tau$.  In that case let $Z_0:=Z+(2j+1)B_m$, $Z_0^-:=Z+2jB_m$ and $Z_0^+:=Z_0+(2j+2)B_m$ which are all symmetric neighbourhoods of the identity and have $Z_0^-+B_m = Z_0$ and $Z_0+B_m =Z_0^+$.  Moreover, the choice of $j$ ensures that $|Z_0^+| \leq (1+\tau)|Z_0^-|$ and the result is proved.
\end{proof}
In particular, for low-dimensional systems we have the following.
\begin{lemma}\label{lem.cc}
Suppose that $B \in \mathcal{S}(G)$ is $d$-dimensional; $\tau \in (0,1]$ is a parameter.  Then there is a set $Z_0 \in \mathcal{N}(G)$ with $B_1\subset Z_0 \subset B_0$ and a natural $m=\log_2d+\log_2\tau^{-1} + O(1)$ such that $(Z_0,B_m)$ is $\tau$-closed. 
\end{lemma}
\begin{proof}
By Lemma \ref{lem.bcn} (\ref{cpt.cs}) and the second inequality in Lemma \ref{lem.cnum} (\ref{pt.equiv}) we have
\begin{equation*}
|B_1+B_1| \leq |B_0| \leq \mathcal{C}(B_0;B_1)|B_1| \leq 2^d|B_1|.
\end{equation*}
Apply Lemma \ref{lem.ubreg} to $Z:=B_1$ and the system $2^{-1}B$ to get the result.
\end{proof}

\section{Proof of Proposition \ref{prop.g}}\label{sec.pf}

It is convenient to use the language of probability theory, but all our probability measures will have finite support so there is no analysis involved.  (This can be easily checked as the only places where we produce new probability spaces are in Lemma \ref{lem.gs} and Lemma \ref{lem.newprob}.)  We say that a probability space $(\Omega',\P')$ is an \textbf{extension} of a probability space $(\Omega,\P)$ if there is a map $\phi:\Omega' \rightarrow \Omega$ such that $\P'(\phi^{-1}(A))=\P(A)$ for all (measurable) $A$ in $\Omega$.  Note that if $(\Omega'',\P'')$ is an extension of $(\Omega',\P')$ and $(\Omega',\P')$ is an extension of $(\Omega,\P)$ then $(\Omega'',\P'')$ is an extension of $(\Omega,\P)$.

Given a random variable $X$ on $\Omega$, and an extension $(\Omega',\P')$ of $(\Omega,\P)$, then for convenience we also write $X$ for the pull-back of $X$ on $\Omega'$. 

We follow the plan in \S\ref{sec.mod}; the analogue of \ref{step1} and Lemma \ref{lem.freiman} is following lemma proved in \S\ref{sec.st}.
\begin{lemma}\label{lem.gs}
Suppose that $A$ is $(k,X)$-summing; $|X| \leq K|A|$; and $\tau \in \left(0,\frac{1}{2}\right]$ is a parameter.  Then there is a $(kK)^{-O(1)}$-hereditarily energetic set $S$ with $|A\cap S| \geq (kK)^{-O(1)}|S|$ and $|S| \geq k^{-O(1)}|A|$, a probability space $(\Omega,\P)$ supporting a $G$-valued random variable $z$, an $\mathcal{N}(G)$-valued random variable $Z$ and an $\mathcal{S}(G)$-valued random variable $T$ such that
\begin{equation*}
\|\E{m_{z+Z}}-m_S\| \leq \tau,
\end{equation*}
and for all $\omega \in \Omega$,
\begin{equation*}
\dim T(\omega) \leq (kK)^{O(1)},\text{ and }\mathcal{C}^\flat(S;T(\omega)) \leq \exp((kK\log \tau^{-1})^{O(1)})
\end{equation*}
and $(Z(\omega),T(\omega)_0)$ is $\tau$-closed.
\end{lemma}
The most technically demanding part is the following analogue of \ref{step3} and Corollary \ref{cor.y} which we prove in \S\ref{sec.en}.
\begin{corollary}\label{cor.iju}
Suppose that $G$ has no $2$-torsion; $A,S \subset G$; a probability space $(\Omega,\P)$ supporting a $G$-valued random variable $z$, an $\mathcal{N}(G)$-valued random variable $Z$ and an $\mathcal{S}(G)$-valued random variable $T$ such that for all $\omega \in \Omega$,
\begin{equation*}
\dim T(\omega) \leq d \text{ and } \mathcal{C}^\flat(S;T(\omega)) \leq D
\end{equation*}
and $(Z(\omega),T(\omega)_0)$ is $\tau$-closed; and $\delta \in (0,1]$ and $r \in \N$ are parameters.  Then either $\tau^{-1}\leq (\delta^{-1}r)^{O(1)}$; or there is a probability space $(\Omega',\P')$ extending $(\Omega,\P)$, supporting a $G$-valued random variable $z'$ and $\mathcal{N}(G)$-valued random variable $Z'$ with
\begin{equation*}
\|\E'{m_{z'+Z'}} - \E{m_{z+Z}}\| \leq\delta,
\end{equation*}
and a further extension $(\Omega'',\P'')$ of $(\Omega',\P')$, supporting a $G$-valued random variable $z''$ and $\mathcal{N}(G)$-valued random variables $Z_1,\dots,Z_k$ such that
\begin{enumerate}
\item 
\begin{equation*}
\|\E''{m_{2^sz''+Z_i}} - \E'{m_{2^s\cdot(z'+Z')}}\| \leq \delta
\end{equation*}
for all $1 \leq i\leq k$ and $0 \leq s \leq r$;
\item  \emph{($U_1$-uniformity)}
\begin{equation*}
\E''{|m_{2^sz''+Z_i}(A)-m_{2^s\cdot(z'+Z')}(A)|^2} \leq \delta
\end{equation*}
for all $1 \leq i\leq k$ and $0 \leq s \leq r$;
\item \emph{($U_2$-uniformity)}
\begin{equation*}
\E''{\left\|1_{A\cap (2^sz''+Z_i)} \ast (1_{A}dm_{2^sz''+Z_j})- m_{2^s\cdot(z'+Z')}(A)^2\right\|_{L_2\left(m_{2^{s+1}z''+Z_i}\right)}^2}\leq \delta
\end{equation*}
for all $1 \leq i <j \leq k$ and $0 \leq s \leq r$;
\end{enumerate}
and for all $\omega'' \in \Omega''$, $(Z_i(\omega''),Z_{i+1}(\omega''))$ is $\delta$-closed for all $1 \leq i < k$, and $\mathcal{C}^\flat(S;Z_k(\omega'')) \leq D\exp((dkr\delta^{-1})^{O(1)})$.
\end{corollary}

Finally, we have the analogue of \ref{step4} and Lemma \ref{lem.ct} which is not terribly different to the model case and which is proved in \S\ref{sec.k}.
\begin{lemma}\label{lem.c}
Suppose that $A,X \subset G$; $z_0 \in G$; $(Z_i,Z_{i+1})$ is $\tau$-closed for all $1 \leq i < k$; and
\begin{enumerate}
\item \label{hyp.1} $|m_{Z_i}(A-z_0) -\alpha|\leq \tau$ for $1 \leq i \leq k$;
\item \label{hyp.2} $m_{Z_i}(X-2z_0) \leq \epsilon$ for all $1 \leq i < k$;
\item \label{hyp.3}
\begin{equation*}
\left\|1_{(A-z_0)\cap Z_i} \ast (1_{A-z_0}dm_{Z_j}) - \alpha^2\right\|_{L_2(m_{Z_i})}^2 \leq \delta \text{ for all }1\leq i <j \leq k.
\end{equation*}
\end{enumerate}
Then either $\delta^{-1}=O(k^2\alpha^{-4})$; or $\tau^{-1} = O(k\alpha^{-1})$; or $\epsilon^{-1}=O(k^2)$; or
\begin{equation*}
\int{\left(\prod_{i<j}{1_{(Z_i+2z_0)\setminus X}(z_i+z_j+2z_0)1_{Z_i}(z_i+z_j)}\right)\prod_{i=1}^k{1_A(z_i+z_0)dm_{Z_i}(z_i)}} =\Omega\left(\alpha^k\right).
\end{equation*}
\end{lemma}
With these tools we can proceed with the main proof.  (We begin by recalling the statement for convenience.)
\begin{proposition*}[Proposition \ref{prop.g}]
Suppose that $G$ has no $2$-torsion; $A,X \subset G$; $|X\setminus A| \leq \eta|A|$; $A$ is $(k,X)$-summing; and $r \in \N$ is a parameter.  Then either $\eta^{-1} = (kr)^{O(1)}$; or $|A| \leq \exp((kr)^{O(1)})$; or there is a $k^{-O(1)}$-hereditarily energetic set $S$ with $|S| \geq k^{-O(1)}|A|$ and $|(2^j \cdot S)\cap A| \geq k^{-O(1)}|S|$ for all $1 \leq j \leq r$.
\end{proposition*}
\begin{proof}
We begin by making a number of choices for parameters.
\begin{itemize}
\item Let $\epsilon_0 = \Omega(k^{-2})$ by such that the third conclusion of Lemma \ref{lem.c} does not follow in any application of that lemma with parameters $\epsilon_0$ and $k$.
\item Let $\nu_0=k^{-O(1)}$ be such that $m_S(A) \geq \nu_0$ and $|S| \geq \nu_0|A|$ in the conclusion of any application of Lemma \ref{lem.gs} with $K =2$.
\item Let $\delta_1= k^{-O(1)}$ be such that the first conclusion in Lemma \ref{lem.c} does not follow in any application of that lemma with parameters $\alpha \geq \min\left\{\frac{1}{2}\nu_0,\frac{1}{2}\epsilon_0\right\}$ and $k$.
\item Let $\frac{1}{4}\epsilon_0 \geq \tau_1 =k^{-O(1)}$ be such that the second conclusion of Lemma \ref{lem.c} does not follow in any application of that lemma with parameters $\alpha \geq \min\left\{\frac{1}{2}\nu_0,\frac{1}{2}\epsilon_0\right\}$ and $k$.
\item Let $\delta_0=(kr)^{-O(1)}$ be such that
\begin{equation*}
(1+4\epsilon_0^{-1}rk + (r+1)k\tau_1^{-2} + (r+1)k^2\delta_1^{-1})\delta_0 \leq \frac{\nu_0^2}{8}.
\end{equation*}
\item Finally, let $\frac{\nu_0^2}{8(4rk\epsilon_0^{-1}+1)} \geq \tau_0= (kr)^{-O(1)}$ be such that the first conclusion of Corollary \ref{cor.iju} does not follow in any application of that corollary with parameters $\delta_0$ and $r$.
\end{itemize}
Apply Lemma \ref{lem.gs} with $K=2$ and parameter $\tau_0$ to get a $k^{-O(1)}$-hereditarily energetic set $S$ with $|A\cap S| \geq \nu_0 |S|$ and $|S| \geq k^{-O(1)}|A|$, and a probability space $(\Omega,\P)$ supporting random variables $z$, $Z$, and $T$ such that
\begin{equation}\label{eqn.lkj}
\|\E{m_{z+Z}} - m_S\| \leq \tau_0.
\end{equation}
and for all $\omega \in \Omega$,
\begin{equation*}
(Z(\omega),T(\omega)_0) \text{ is $\tau_0$-closed, } \dim T(\omega) \leq k^{O(1)}, \mathcal{C}^\flat(S;T(\omega)_0) \leq \exp((kr)^{O(1)}).
\end{equation*}
Apply Corollary \ref{cor.iju} with parameters $\delta_0$ and $r$, so that (in light of the choice of $\tau_0$) we get an extension $(\Omega',\P')$ supporting random variables $z'$, $Z'$,  such that
\begin{equation*}
\|\E'{m_{z'+Z'}} - \E{m_{z+Z}} \| \leq \delta_0,
\end{equation*}
and a further extension $(\Omega'',\P'')$  supporting $z''$, $Z_1,\dots,Z_k$ such that 
\begin{equation}\label{eqn.cx}
\|\E''{m_{2^sz''+Z_i}} - \E'{m_{2^s\cdot(z'+Z')}}\| \leq \delta_0
\end{equation}
for all $1 \leq i\leq k$ and $0 \leq s \leq r$,
\begin{equation}\label{eqn.dwe}
\E''{|m_{2^sz''+Z_i}(A)-m_{2^s\cdot(z'+Z')}(A)|^2} \leq \delta_0
\end{equation}
for all $1 \leq i\leq k$ and $0 \leq s \leq r$, and
\begin{equation}\label{eqn.frt}
\E''{\left\|1_{A\cap (2^sz''+Z_i)} \ast (1_{A}dm_{2^sz''+Z_j})- m_{2^s\cdot(z'+Z')}(A)^2\right\|_{L_2\left(m_{2^{s+1}z''+Z_i}\right)}^2}\leq \delta_0
\end{equation}
for all $1 \leq i <j \leq k$ and $0 \leq s \leq r$, and for all $\omega'' \in \Omega''$ we have
\begin{equation*}
(Z_i(\omega''),Z_{i+1}(\omega'')) \text{ is $\delta_0$-closed for all $1 \leq i <k$}
\end{equation*}
and
\begin{equation*}
\mathcal{C}^\flat(S;Z_k(\omega'')) \leq \exp((kr)^{O(1)})\mathcal{C}^\flat(S;T(\omega'')) = \exp((kr)^{O(1)}).
\end{equation*}
For each $0 \leq s <r$ and $1 \leq i<j\leq k$ write
\begin{align*}
E_{s,i,j}&:=\left\{\omega'' \in \Omega'': \left\|1_{A\cap (2^sz''(\omega'')+Z_i(\omega''))} \ast (1_{A}dm_{2^sz''(\omega'')+Z_j(\omega'')})\right.\right.\\
&\qquad \qquad \qquad \qquad \qquad \left.\left.- m_{2^s\cdot(z'(\omega'')+Z'(\omega''))}(A)^2\right\|_{L_2\left(m_{2^{s+1}z''(\omega'')+Z_i(\omega'')}\right)}^2 \leq \delta_1\right\},
\end{align*}
so that by averaging and (\ref{eqn.frt}) we have $\P''(E_{s,i,j}^c) \leq \delta_1^{-1}\delta_0$.  For $0\leq s \leq r$ and $1\leq i \leq k$ write
\begin{equation*}
F_{s,i}:=\left\{\omega'' \in \Omega'': |m_{2^sz''(\omega'')+Z_i(\omega'')}(A)-m_{2^s\cdot(z'(\omega'')+Z'(\omega''))}(A)| \leq \tau_1\right\},
\end{equation*}
so that by averaging and (\ref{eqn.dwe}) we have $\P''(F_{s,i}^c) \leq \tau_1^{-2}\delta_0$.  For each $1 \leq s \leq r$ and $1 \leq i \leq k$ write
\begin{equation*}
L_{s,i}:=\left\{\omega'' \in \Omega'': m_{2^sz''(\omega'')+Z_i(\omega'')}(X \setminus A) \leq \frac{1}{4}\epsilon_0\right\}.
\end{equation*}
Then by (\ref{eqn.cx}) and (\ref{eqn.lkj}) we have
\begin{align*}
\P''(L_{s,i}^c) &\leq 4\epsilon_0^{-1}\E''{m_{2^sz''(\omega'')+Z_i(\omega'')}(X \setminus A)}\\ & \leq 4\epsilon_0^{-1}(\delta_0 +  \E'{m_{2^s\cdot(z'+Z')}(X\setminus A)})\\
& \leq 4\epsilon_0^{-1}(\delta_0 + \tau_0 + m_S(X \setminus A)) \leq 4\epsilon_0^{-1}(\delta_0+\tau_0+\nu_0^{-1}\eta).
\end{align*}
Finally, let
\begin{equation*}
B:=\left\{\omega'' \in\Omega'': m_{z'(\omega'')+Z'(\omega'')}(A) \geq \frac{1}{2}\nu_0\right\}
\end{equation*}
and
\begin{equation*}
 R:=B \cap \left(\bigcap_{s=1}^r{\bigcap_{i=1}^k{L_{s,i}}}\right)\cap \left(\bigcap_{s=0}^r{\bigcap_{i=1}^k{F_{s,i}}}\right)\cap \left(\bigcap_{s=0}^{r-1}{\bigcap_{1 \leq i<j\leq k}{E_{s,i,j}}}\right).
\end{equation*}
Then we have
\begin{align}
\label{eqn.fjt} \E''{1_Rm_{z'+Z'}(A)} &\geq \E''{m_{z'+Z'}(A)} \\
\nonumber &\qquad \qquad - \sum_{s=1}^r{\sum_{i=1}^k{\P''(L_{s,i}^c)}} - \sum_{s=0}^r{\sum_{i=1}^k{\P''(F_{s,i}^c)}}-\sum_{s=0}^{r-1}{\sum_{1 \leq i<j\leq k}{\P''(E_{s,i,j}^c)}}\\
\nonumber &\geq \left(m_S(A) - \|\E''{m_{z'+Z'}} - \E{m_{z+Z}}\| - \|\E{m_{z+Z}} - m_S\|\right)\\
\nonumber &\qquad \qquad -rk\cdot  4\epsilon_0^{-1}(\delta_0+\tau_0+\nu_0^{-1}\eta)  - (r+1)k\tau_1^{-2}\delta_0-(r+1)k^2\delta_1^{-1}\delta_0\\
\nonumber &\geq \nu_0-\delta_0-\tau_0\\
\nonumber &\qquad \qquad -rk\cdot  4\epsilon_0^{-1}(\delta_0+\tau_0+\nu_0^{-1}\eta)  - (r+1)k\tau_1^{-2}\delta_0-(r+1)k^2\delta_1^{-1}\delta_0.
\end{align}
This is at least $\frac{1}{2}\nu_0$ by choice of $\delta_0$ and $\tau_0$ provided $\eta \leq \frac{\nu_0^2\epsilon_0}{16rk}$ (if not then we are in the first conclusion of the proposition).
\begin{claim*}
Either $|A| \leq \exp((kr)^{O(1)})$ or else for all $\omega'' \in E$ we have\footnote{It may help to note that apart from the $s=0$ case the lower bound we establish is always $\frac{1}{2}\epsilon_0$.}
\begin{equation*}
m_{2^s\cdot (z'(\omega'')+Z'(\omega''))}(A)\geq\min\left\{\frac{1}{2}\nu_0,\frac{1}{2}\epsilon_0\right\} \text{ for all }0 \leq s \leq r.
\end{equation*}
\end{claim*}
\begin{proof}
We proceed by induction on $s$.  When $s=0$ we have $\omega'' \in B$ and so the hypothesis holds.  Suppose that it holds for some $s <r$.  Then since $\omega'' \in E_{s,i,j}$ for all $1 \leq i<j \leq k$ we have
\begin{align*}
&\big\|1_{A\cap (2^sz''(\omega'')+Z_i(\omega''))} \ast (1_{A}dm_{2^sz''(\omega'')+Z_j(\omega'')})\\ & \qquad \qquad \qquad \qquad \qquad \qquad- m_{2^s\cdot(z'(\omega'')+Z'(\omega''))}(A)^2\big\|_{L_2\left(m_{2^{s+1}z''(\omega'')+Z_i(\omega'')}\right)}^2 \leq \delta_1,
\end{align*}
and since $\omega \in F_{s,i}$ for all $1 \leq i \leq k$ we have
\begin{align*}
 |m_{2^sz''(\omega'')+Z_i(\omega'')}(A)-m_{2^s\cdot(z'(\omega'')+Z'(\omega''))}(A)| \leq \tau_1.
\end{align*}
Apple Lemma \ref{lem.c} with parameters $\epsilon_0$, $\delta_1$, $\tau_1$ and $\alpha=m_{2^s\cdot (z'(\omega'')+Z'(\omega''))}(A)$ to the set $A$ with $z_0=2^sz''(\omega'')$ and the $\delta_0$-closed pairs $(Z_i(\omega''),Z_{i+1}(\omega''))$ (for $1 \leq i <k$) valid since $\delta_0 \leq \tau_1$.

The inductive hypothesis tells us that $\alpha \geq \min\left\{\frac{1}{2}\nu_0,\frac{1}{2}\epsilon_0\right\}$, so in view of the choice of $\delta_1$ and $\tau_1$, either there is some $1 \leq i< k$ such that $m_{2^{s+1}z''(\omega'')+Z_i(\omega'')}(X)>\epsilon_0$ or
\begin{align*}
&\int{\left(\prod_{i<j}{1_{(Z_i+2^{s+1}z''(\omega''))\setminus X}(z_i+z_j+2^{s}z''(\omega''))1_{Z_i(\omega'')}(z_i+z_j)}\right)}\\ &\qquad \qquad \qquad \qquad \qquad \qquad \qquad \times \prod_{i=1}^k{1_A(z_i+2^sz''(\omega''))dm_{Z_i(\omega'')}(z_i)}\\ & \qquad \qquad\qquad \qquad \qquad \qquad \qquad \qquad \qquad \qquad=\Omega\left(\left(\min\left\{\frac{1}{2}\nu_0,\frac{1}{2}\epsilon_0\right\}\right)^k\right).
\end{align*}
Since $A$ is $(k,X)$-summing we know that the first product on the left is $0$ on $z \in A^k$ unless there is $1 \leq i < j \leq k$ with $z_i=z_j$.  It follows that the integral is at most 
\begin{align*}
\binom{k}{2}\cdot \max\left\{\frac{1}{|Z_i(\omega'')|}:1 \leq i \leq k\right\} &\leq k^2\max\left\{\frac{\mathcal{C}(S;Z_i(\omega''))}{|S|}:1 \leq i \leq k \right\}\\ & \leq k^2\frac{1}{\nu_0 |A|}\mathcal{C}^\flat(S;Z_k(\omega'')) \leq \exp((rk)^{O(1)})|A|^{-1}
\end{align*}
where the first inequality follows from Lemma \ref{lem.bcn} (\ref{cpt.cs}); and the second from Lemma \ref{lem.cnum} (\ref{pt.order}) and (\ref{pt.equiv}) using the nesting $Z_k(\omega'') \subset Z_i(\omega'')$.

The upper bound on $|A|$ follows.

Thus we may assume $m_{2^{s+1}z''(\omega'')+Z_i(\omega'')}(X)>\epsilon_0$, and since $\omega \in L_{s+1,i}$ we have
\begin{align*}
m_{2^{s+1}z''(\omega'')+Z_i(\omega'')}(A) & \geq m_{2^{s+1}z''(\omega'')+Z_i(\omega'')}(X) - m_{2^{s+1}z''(\omega'')+Z_i(\omega'')}(X \setminus A) > \frac{3}{4}\epsilon_0.
\end{align*}
Finally the claim is proved once we note that $\omega'' \in F_{s+1,i}$ on noting that $\tau_1 \leq \frac{1}{4}\epsilon_0$.
\end{proof}

We conclude that for all $1 \leq s \leq r$ we have
\begin{align*}
m_{2^s\cdot S}(A) &\geq \E''{m_{2^s\cdot (z'+Z')}(A)} - \|m_{2^s\cdot S} - \E''{m_{2^s\cdot (z'+Z')}}\|\\
& \geq \E''{1_Rm_{2^s\cdot (z'+Z')}(A)}- \delta_0-\tau_0\\
& \geq \E''{1_Rm_{z'+Z'}(A)m_{2^s\cdot (z'+Z')}(A)} - \delta_0-\tau_0\\
& \geq \min\left\{\frac{1}{2}\nu_0,\frac{1}{2}\epsilon_0\right\}\cdot \frac{1}{2}\nu_0 - \delta_0-\tau_0
\end{align*}
from the lower bound on the left of (\ref{eqn.fjt}) noted immediately afterwards.  The result follows since $\delta_0,\tau_0 \leq \frac{1}{16}\min\{\nu_0^2,\nu_0\epsilon_0\}$.
\end{proof}

\section{Finding structure}\label{sec.st}

In this section we prove the following.
\begin{lemma*}[Lemma \ref{lem.gs}]
Suppose that $A$ is $(k,X)$-summing; $|X| \leq K|A|$; and $\tau \in \left(0,\frac{1}{2}\right]$ is a parameter.  Then there is a $(kK)^{-O(1)}$-hereditarily energetic set $S$ with $|A\cap S| \geq (kK)^{-O(1)}|S|$ and $|S| \geq k^{-O(1)}|A|$, a probability space $(\Omega,\P)$ supporting a $G$-valued random variable $z$, an $\mathcal{N}(G)$-valued random variable $Z$ and an $\mathcal{S}(G)$-valued random variable $T$ such that
\begin{equation*}
\|\E{m_{z+Z}}-m_S\| \leq \tau,
\end{equation*}
and for all $\omega \in \Omega$,
\begin{equation*}
\dim T(\omega) \leq (kK)^{O(1)},\text{ and }\mathcal{C}^\flat(S;T(\omega)) \leq \exp((kK\log \tau^{-1})^{O(1)})
\end{equation*}
and $(Z(\omega),T(\omega)_0)$ is $\tau$-closed.
\end{lemma*}
Recall from \cite[Definition 0.2]{taovu::} that $P$ is a \textbf{generalised arithmetic progression of dimension $d$} if there are elements $x_0,\dots,x_d$ and naturals $N_1,\dots,N_d\in \N$ such that
\begin{equation}\label{eqn.rep}
P=\{x_0+t_1x_1+\cdots + t_dx_d : 0 \leq t_j \leq N_j \text{ for all }1 \leq j \leq d\}.
\end{equation}
A \textbf{coset progression of dimension $d$} is a set of the form $P+H$ where $P$ is a generalised arithmetic progression of dimension $d$ and $H \leq G$.  This corresponds to \cite[Definition 4.20]{taovu::} where dimension is given the name rank instead.

Coset progressions are easily related to systems as follows.
\begin{lemma}\label{lem.gap}
Suppose that $M$ is a $d$-dimensional coset progression.  Then there is an $O(d)$-dimensional system $B$ with $B_0\subset M-M$ and $M \subset x_0 +B_0$ for some $x_0$. 
\end{lemma}
\begin{proof}
Write $M=P+H$ where $P$ is as in (\ref{eqn.rep}) and put
\begin{equation*}
B_i:=\{t_1x_1+\cdots + t_dx_d : t_j \in \Z \text{ and }|t_j| \leq 2^{-i}N_j \text{ for all }1 \leq j \leq d\} +H
\end{equation*}
so that $B_0\subset M-M$ and $M \subset x_0+B_0$.  $B_i+B_i \subset B_{i-1}$ by the triangle inequality and each set is a symmetric neighbourhood of $0_G$, so $B=(B_i)_{i \in \N_0}$ is a system.  Moreover,
\begin{equation*}
B_i \subset  B_{i+1}+\{\sigma_1\lceil 2^{-(i+1)}N_1\rceil x_1+\cdots + \sigma_d\lceil 2^{-(i+1)}N_d\rceil x_d: \sigma_j \in \{-1,0,1\}\text{ for all } 1 \leq j \leq d\},
\end{equation*}
so $\mathcal{C}(B_i;B_{i+1}) \leq 3^d$.  It follows (\emph{c.f.} the proof of Lemma \ref{lem.new} (\ref{pt.szy})) that
\begin{align*}
\mathcal{C}^\flat(B_i;B_{i+1}) &\leq \mathcal{C}^\flat(B_i;B_{i+2}-B_{i+2})\\ & \leq \mathcal{C}(B_i;B_{i+2}) \leq \mathcal{C}(B_i;B_{i+1})\mathcal{C}(B_{i+1};B_{i+2})\leq \exp(O(d)),
\end{align*}
and we see that $B$ is $O(d)$-dimensional.
\end{proof}

\begin{proof}[Proof of Lemma \ref{lem.gs}]
By Lemma \ref{lem.hen} we see that $E(A) \geq (kK)^{-O(1)}|A|^3$ and so by the Balog-Szemer{\'e}di-Gowers Theorem\footnote{In its usual form, which corresponds to \cite[Theorem 2.31((i) {$\Rightarrow$} (iv))]{taovu::} and then \cite[Exercise 2.3.15]{taovu::}).} there is a set $A' \subset A$ with $|A'+A'| \leq (kK)^{O(1)}|A'|$ and $|A'| = (kK)^{-O(1)}|A|$.  Apply the Ruzsa-Chang theorem \cite[Theorem 5.46]{taovu::} to see that $2A'-2A'$ contains a coset progression $M$ of dimension $(kK)^{O(1)}$ such that $|M| \geq \exp(-(kK)^{O(1)})|A|$.  By Lemma \ref{lem.gap} there is a system $B$ such that $\dim B =(kK)^{O(1)}$, $B_0 \subset M-M \subset 4A'-4A'$ and $|B_0| \geq |M| \geq \exp(-(kK)^{O(1)})|A|$.

Since $|A'-A'+B_2| \leq |A'-A'+B_0| \leq |5A'-5A'| \leq (kK)^{O(1)}|A'-A'|$ by Pl{\"u}nnecke's inequality \cite[Corollary 6.27]{taovu::}, we can apply Lemma \ref{lem.ubreg} to the system $2^{-2}B$ and set $A'-A' \in \mathcal{N}(G)$ to get $m=\log_2 \log_2 kK +\log_2\tau^{-1}+ O(1)$ and a set $S \in \mathcal{N}(G)$ with $A'-A' \subset S \subset A'-A'+B_2$ so that $(S,B_{m+2})$ is $\tau$-closed.  By Lemma \ref{lem.basicclosed} (\ref{pt.inv}) it follows that
\begin{equation*}
\|\tau_t(m_S) - m_S\| \leq \tau \text{ for all }t \in B_{m+2}.
\end{equation*}
Apply Lemma \ref{lem.cc} to the system $2^{-(m+2)}B$ (which has dimension $(kK)^{O(1)}$ by Lemma \ref{lem.new} (\ref{pt.szy})) to get a natural $m' = O(\log_2 kK) +\log_2\tau^{-1}$ and a set $Z$ with $B_{m+3} \subset Z_* \subset B_{m+2}$ such that $(Z_*,B_{m'+m+2})$ is $\tau$-closed.

Let $z$ be the random variable taking values in $S$ uniformly; let $Z$ be the constant random variable taking the value $Z_*$; and let $T$ be the constant random variable taking the value $2^{-(m'+m+2)}B$.  Then
\begin{equation*}
\|\E{m_{z+Z}} - m_S\| = \left\|m_S \ast m_{Z_*} -m_S\right\| \leq \int{\|\tau_t(m_S) - m_S\|dm_{Z_*}(t)} \leq \tau.
\end{equation*}
Since $S \subset A'-A'+B_2$ we have $|S| \leq (kK)^{O(1)}|A|$; since $S \supset A'-A'$, $S$ intersects a translate of $A$ in a set of size at least $|A'| =(kK)^{-O(1)}|A|=(kK)^{-O(1)}|S|$; by translating $A$ if necessary (which results in translations of $z$ and $S$ too) we can assume that this translate is at $0_G$ so that $|A\cap S| \geq (kK)^{-O(1)}|S|$.  Since $S \subset A'-A'+B_2 \subset 5A'-5A'$ we see that $|S+S| \leq |10A'-10A'| \leq (kK)^{O(1)}|S|$ by Pl{\"u}nnecke's inequality \cite[Corollary 6.27]{taovu::}.  It follows from Lemma \ref{lem.hed} (\ref{pt.1}) that $S$ is $(kK)^{-O(1)}$-hereditarily energetic.

Finally we let $T$ be the random variable taking the constant value $2^{-(m+m'+2)}B$.  Then $(Z(\omega),T(\omega)_0)$ is $\tau$-closed by design; Lemma \ref{lem.new} (\ref{pt.szy}) tells us that
\begin{equation*}
\dim T(\omega) = \dim 2^{-(m+m'+2)}B\leq \dim B =(kK)^{O(1)},
\end{equation*}
and also
\begin{equation*}
\mathcal{C}^\flat(S;T(\omega)) \leq \exp(O((m+m'+2))(kK)^{O(1)})\mathcal{C}^\flat(S;B) = \exp((kK\log\tau^{-1})^{O(1)}\mathcal{C}^\flat(S;B).
\end{equation*}
Now, by Lemma \ref{lem.bcn} (\ref{cpt.cs}) and (\ref{cpt.chain}), and the second inequality in Lemma \ref{lem.cnum} (\ref{pt.equiv}) we have
\begin{equation*}
\frac{|B_0|}{|B_2|} \leq \mathcal{C}(B_0;B_2) \leq \mathcal{C}(B_0;B_1)\mathcal{C}(B_1;B_2) \leq \mathcal{C}^\flat(B_0;B_1)\mathcal{C}^\flat(B_1;B_2) \leq \exp((kK)^{O(1)}),
\end{equation*}
so $|B_2| \geq \exp(-(kK)^{O(1)})|A|$.  By Lemma \ref{lem.cnum} (\ref{pt.order}) and (\ref{pt.equiv}), and then Ruzsa's covering lemma (Lemma \ref{lem.bcn} (\ref{cpt.rcl})) we have 
\begin{align*}
\mathcal{C}^\flat(S;B)=\mathcal{C}^\flat(S;B_0) &\leq \mathcal{C}^\flat(S;(B_2-B_2)-(B_2-B_2))\\ & \leq \mathcal{C}(S;B_2-B_2)\leq \frac{|S+B_2|}{|B_2|} \leq \frac{|A'-A'+B_2+B_2|}{|B_2|}.
\end{align*}
By Pl{\"u}nnecke's inequality \cite[Corollary 6.27]{taovu::} and the lower estimate on $|B_2|$ we have $\mathcal{C}^\flat(S;B)\leq \exp((kK)^{O(1)})$, and the result is proved.
\end{proof}

\section{Uniformity}\label{sec.en}

The aim of this section is to prove Corollary \ref{cor.iju}  We shall do this through the Fourier transform which we introduced in \S\ref{sec.over}.  The reader unfamiliar with its use in this context may wish to consult \cite[Chapter 4]{taovu::} or the book \cite{rud::1}.

We begin with some definitions: the Fourier transform of a measure $\mu \in M(G)$ (where $M(G)$ denotes the set of measures on $G$ with finite support) is
\begin{equation*}
\wh{\mu}(\gamma):=\int{\overline{\gamma(x)}d\mu(x)} \text{ for all }\gamma \in \wh{G},
\end{equation*}
and given $f \in C(G)$ we define
\begin{equation*}
\mu \ast f(x) = f \ast \mu(x) = \int{fd\tau_x(\mu)} \text{ for all }x \in G.
\end{equation*}
Given $\mu \in M(G)$ and a parameter $\epsilon >0$ we write (\emph{c.f.} \cite[Definition 4.33]{taovu::})
\begin{equation}\label{eqn.spectrum}
\Spec_\epsilon(\mu):=\{\gamma \in \wh{G}: |\wh{\mu}(\gamma)| > \epsilon\}.
\end{equation}
Motivated by \cite[Theorem 1.2.6]{rud::1}, for $W \subset G$ and $\epsilon>0$ a parameter we write
\begin{equation*}
\Ann(W,\epsilon):=\{\gamma \in \wh{G}: |\gamma(x)-1|<\epsilon \text{ for all }x \in W\}.
\end{equation*}
The spectrum and approximate annihilators of $\eta$-closed pairs are closely related by the next result which is \cite[Lemma 3.6]{grekon::}.
\begin{lemma}\label{lem.inv}
Suppose that $(Z,W)$ is $\eta$-closed; and $\kappa \in (0,1]$ is a parameter.  Then $\Spec_\kappa(m_Z) \subset \Ann(W, \eta\kappa^{-1})$.
\end{lemma}
\begin{proof}
For $x \in W$ simply note that
\begin{align*}
\kappa |\gamma(x)-1| &< |\gamma(x)\wh{m_Z}(\gamma) -\wh{m_Z}(\gamma)|\\ & = \left|\int{\overline{\gamma(y)}d(\tau_x(m_Z)-m_Z)(y)}\right| \leq \|\tau_x(m_Z)-m_Z\| \leq \eta.
\end{align*}
\end{proof}
Dually to approximate annihilators we have Bohr sets which provide us with a ready supply of low dimensional systems.
\begin{lemma}[Bohr sets]\label{lem.bohr}
Suppose that $\gamma \in \wh{G}$ and $\rho>0$.  Then there is a system $B$ with $\dim B =O(1)$, $\mathcal{C}^\flat(G;B_0) =O(\rho^{-O(1)})$, and
\begin{equation*}
|\gamma(x)-1|<\rho \text{ for all }x \in B_0.
\end{equation*}
\end{lemma}
\begin{proof}
Let $A_i:=\{z \in S^1: |1-z|<2^{2-i}\}$ so that $A_{i+1}-A_{i+1} \subset A_{i}$ by the triangle inequality.  Let $T_i:=\{z \in S^1: |1-z| \in \{0,2^{-i},2\cdot 2^{-i},3\cdot 2^{-i},4\cdot 2^{-i}\}\}$ so that $|T_i|\leq 9$ and $A_i \subset T_i +A_{i+2}$, whence $\mathcal{C}(A_i;A_{i+2}) =O(1)$.  Put $B_i'=\gamma^{-1}(A_i)$ for all $i \in \N_0$.  By the definition of $\mathcal{C}^\flat$ we have
\begin{equation*}
\mathcal{C}^\flat(B_i';B_{i+1}') \leq \mathcal{C}(A_i;A_{i+2})=O(1),
\end{equation*}
so $B'$ is $O(1)$-dimensional.  On the other hand there is $m=\log_2\rho^{-1} + O(1)$ such that $B_m' \subset \{x: |\gamma(x)-1| <\rho\}$; setting $B:=2^{-m}B$ gives the result by Lemma \ref{lem.new} (\ref{pt.szy}).
\end{proof}
The Parseval bound is the standard way to bound the size of spectrum (\emph{c.f.} \cite[(4.38)]{taovu::}).  In our approximate setting, we do not have perfectly orthogonal characters necessary for Parseval's theorem, but fortunately we do have a notion of `almost orthogonal'.  This was first exploited to achieve a result of the below type by Green and Tao -- see \cite[Corollary 8.6]{gretao::1}.
\begin{lemma}[The Parseval bound]\label{lem.pars}
Suppose that $(Z,W)$ is $\tau$-closed; $f\in L_2(m_Z)$ has $\|f\|_{L_2(m_Z)}\leq 1$; and $\epsilon>0$ and $\delta \in \left(0,\frac{1}{2}\right]$ are parameters.  Then there is a system $B$ with $\dim B =O(\epsilon^{-2})$ and $\mathcal{C}^\flat(G;B_0) \leq \delta^{-O(\epsilon^{-2})}$ and
\begin{equation*}
\Spec_\epsilon(fdm_Z) \subset \Ann(B_0 \cap W, 4\epsilon^{-2}\tau+ \delta).
\end{equation*}
\end{lemma}
\begin{proof}
Let $\Lambda \subset \Spec_\epsilon(fdm_Z)$ be a maximal subset such that if $(\lambda+ \Ann(W,2\epsilon^{-2}\tau)) \cap (\lambda'+ \Ann(W,2\epsilon^{-2}\tau)) \neq \emptyset$ for some $\lambda,\lambda' \in \Lambda$ then $\lambda = \lambda'$.  By maximality we have
\begin{equation*}
\Spec_\epsilon(fdm_Z) \subset \Lambda + \Ann(W,2\epsilon^{-2}\tau)-\Ann(W,2\epsilon^{-2}\tau) \subset \Lambda + \Ann(W,4\epsilon^{-2}\tau),
\end{equation*}
where the last inclusion is by the triangle inequality.

On the other hand writing $\sigma_\lambda$ for the sign of $\wh{fdm_Z}(\lambda)$ we can use linearity and the Cauchy-Schwarz inequality to see that
\begin{align*}
|\Lambda|\epsilon \leq \sum_{\lambda \in \Lambda}{|\wh{fdm_Z}(\lambda)|} &= \left\langle f,\sum_{\lambda \in \Lambda}{\sigma_\lambda\lambda}\right\rangle_{L_2(m_Z)} \leq \|f\|_{L_2(m_Z)}\left(\sum_{\lambda, \lambda' \in \Lambda}{|\langle\lambda ,\lambda'\rangle_{L_2(m_Z)}|}\right)^{\frac{1}{2}}.
\end{align*}
By Lemma \ref{lem.inv} we have $\Spec_{\frac{1}{2}\epsilon^2}(m_Z) \subset \Ann(W,2\epsilon^{-2}\tau)$ so
\begin{equation*}
\sum_{\lambda, \lambda' \in \Lambda}{|\langle\lambda ,\lambda'\rangle_{L_2(m_Z)}|} \leq |\Lambda| + \frac{1}{2}\epsilon^2|\Lambda|^2.
\end{equation*}
Rearranging and using the fact that $\|f\|_{L_2(m_Z)}\leq 1$ (by hypothesis) we see that $|\Lambda| \leq 2\epsilon^{-2}$.  For each $\lambda \in \Lambda$ let $B^{(\lambda)}$ be the system given by Lemma \ref{lem.bohr}; let $B:=\bigcap_{\lambda \in \Lambda}{B^{(\lambda)}}$.  Then by Lemma \ref{lem.new} (\ref{pt.intdim}) we see that $\dim B=O(\epsilon^{-2})$, and by Lemma \ref{lem.cnum} (\ref{pt.equiv}) we have $\mathcal{C}^\flat(G;B_0) \leq \delta^{-O(\epsilon^{-2})}$.  By definition $\Lambda \subset \Ann(B_0,\delta)$; the result follows.
\end{proof}
We are now in a position to prove that large local $U_2$-norm gives rise to a density increment.  As a small word of caution we remark that normalising constants below may not be exactly as expected -- this is already apparent from the definition of the spectrum in (\ref{eqn.spectrum}) where we look at characters where $|\wh{\mu}(\gamma)|>\epsilon$ \emph{not} $|\wh{\mu}(\gamma)| > \epsilon \|\mu\|$.  The latter is essentially the definition in \cite[Definition 4.33]{taovu::}.
\begin{lemma}\label{lem.increment}
Suppose that $A \subset G$, $(Z_0,Z_1)$ and $(Z_1,Z_2)$ are $\tau$-closed, $|m_{Z_0}(A)-\alpha|<\tau$ and $|m_{Z_1}(A)-\alpha|<\tau$, and $\epsilon,\delta>0$ are parameters.  Then there is a Bohr system $B$ with $\dim B = O(\epsilon^{-2})$ and $\mathcal{C}^\flat(G;B_0) \leq \delta^{-O(\epsilon^{-2})}$ such that
\begin{equation*}
\|1_{A} \ast m_{Z'}-\alpha\|_{L_2(m_{Z_0})}^2 \geq \|1_{A\cap Z_0}\ast (1_A dm_{Z_1})-\alpha^2\|_{L_2(m_{Z_0})}^2-O(\epsilon^2 + \epsilon^{-2}\tau + \delta)
\end{equation*}
for all $Z' \subset Z_2\cap B_0$.
\end{lemma}
\begin{proof}
First note that
\begin{equation*}
1_{Z_0} \ast (1_A dm_{Z_1})(z)=m_{Z_1}(A)  \text{ for all } z \in Z_0^-,
\end{equation*}
whence
\begin{equation*}
\|1_{Z_0}\ast  (1_A dm_{Z_1}) -\alpha\|_{L_1(m_{Z_0})}\leq |m_{Z_1}(A)-\alpha| + \max\{\alpha,1-\alpha\}m_{Z_0}(Z_0 \setminus Z_0^-) = O(\tau).
\end{equation*}
The inequality $\|f\|_{L_2}^2\leq \|g\|_{L_2}^2 + (\|f\|_{L_\infty} + \|g\|_{L_\infty})\|f-g\|_{L_1}$ and Plancherel's theorem then give
\begin{align*}
 &|Z_0|\|1_{A\cap Z_0}\ast (1_A dm_{Z_1})-\alpha^2\|_{L_2(m_{Z_0})}^2\\
  & \qquad \qquad \qquad \leq |Z_0|\|(1_{A\cap Z_0}-\alpha 1_{Z_0})\ast (1_A dm_{Z_1})\|_{L_2(m_{Z_0})}^2 \\
  & \qquad \qquad \qquad \qquad\qquad \qquad+ O\left(|Z_0|\|1_{Z_0}\ast  (1_A dm_{Z_1}) -\alpha\|_{L_1(m_{Z_0})}\right)\\
 & \qquad \qquad \qquad = |Z_0|\|(1_{A\cap Z_0}-\alpha 1_{Z_0})\ast (1_A dm_{Z_1})\|_{L_2(m_{Z_0})}^2 + O(\tau |Z_0|)\\
& \qquad \qquad \qquad \leq \| (1_{A \cap Z_0} - \alpha1_{Z_0}) \ast (1_A dm_{Z_1})\|_{\ell_2(G)}^2+ O(\tau |Z_0|)\\
& \qquad \qquad \qquad = \int{|(1_{A\cap Z_0} - \alpha 1_{Z_0})^\wedge(\gamma)|^2|(1_Adm_{Z_1})^\wedge(\gamma)|^2d\gamma}+ O(\tau |Z_0|).
\end{align*}
On the other hand
\begin{equation*}
\int{|(1_{A\cap Z_0} - \alpha 1_{Z_0})^\wedge(\gamma)|^2d\gamma} = \|1_{A\cap Z_0}-\alpha 1_{Z_0}\|_{\ell_2(G)}^2 \leq (1+O(\tau))|Z_0|,
\end{equation*}
and so
\begin{align*}
 &|Z_0|\|1_{A\cap Z_0}\ast (1_A dm_{Z_1})-\alpha^2\|_{L_2(m_{Z_0})}^2\\
  & \qquad \qquad\leq \int_{\Spec_{\epsilon}(1_Adm_{Z_1})}{|(1_{A\cap Z_0} - \alpha 1_{Z_0})^\wedge(\gamma)|^2d\gamma} + \epsilon^2|Z_0|+ O(\tau |Z_0|).
\end{align*}
Apply Lemma \ref{lem.pars} (we may certainly assume $\delta \leq \frac{1}{2}$ or else there is nothing to prove) to the $\tau$-closed pair $(Z_1,Z_2)$ and $1_{A}$ (which has $\|1_{A}\|_{L_2(m_{Z_1})} \leq 1$) to get a system $B$ with $\dim B=O(\epsilon^{-2})$ and $\mathcal{C}^\flat(G;B_0) \leq \delta^{-O(\epsilon^{-2})}$ s.t.
\begin{equation*}
\Spec_{\epsilon}(1_Adm_{Z_1}) \subset \Ann(B_0\cap Z_2,4\epsilon^{-2}\tau +\delta).
\end{equation*}
It follows that if $Z' \subset B_0\cap Z_2$ then
\begin{equation*}
|\wh{m_{Z'}}(\gamma)-1|=O(\epsilon^{-2}\tau +\delta) \text{ for all }\gamma\in\Spec_{\epsilon}(1_Adm_{Z_1}),
\end{equation*}
and hence by the triangle inequality
\begin{align*}
&\int{|(1_{A\cap Z_0} - \alpha 1_{Z_0})^\wedge(\gamma)|^2|\wh{m_{Z'}}(\gamma)|^2d\gamma}\\
& \qquad \qquad \qquad \geq \int_{\Spec_{\epsilon}(1_Adm_{Z_1})}{|(1_{A\cap Z_0} - \alpha 1_{Z_0})^\wedge(\gamma)|^2d\gamma} - O((\epsilon^{-2}\tau + \delta)|Z_0|)\\
& \qquad \qquad \qquad \geq |Z_0|\|1_{A\cap Z_0}\ast (1_A dm_{Z_1})-\alpha^2\|_{L_2(m_{Z_0})}^2- O((\epsilon^2+\epsilon^{-2}\tau + \delta)|Z_0|).
\end{align*}
The result follows on dividing by $|Z_0|$ and applying Parseval's theorem again to see that
\begin{align*}
\int{|(1_{A\cap Z_0} - \alpha 1_{Z_0})^\wedge(\gamma)|^2|\wh{m_{Z'}}(\gamma)|^2d\gamma}& =\|(1_{A\cap Z_0} - \alpha 1_{Z_0})\ast m_{Z'}\|_{\ell_2(G)}^2\\
& =\|(1_{A\cap Z_0} - \alpha 1_{Z_0})\ast m_{Z'}\|_{\ell_2(Z_0^-)}^2 + O(|Z_0^+ \setminus Z_0^-|)\\
& =\|1_{A}\ast m_{Z'} - \alpha\|_{\ell_2(Z_0^-)}^2 + O(|Z_0^+ \setminus Z_0^-|)\\
&\leq \|1_{A} \ast m_{Z'}-\alpha\|_{L_2(m_{Z_0})}^2|Z_0| + O(\tau |Z_0|)
\end{align*}
since $Z' \subset Z_2 \subset Z_1$ and $(Z_0,Z_1)$ is $\tau$-closed.
\end{proof}
We now have all the Fourier tools we need.  The next lemma captures some facts about the non-model analogue of weighted covers (from \ref{stepdef} in \S\ref{sec.mod}) in a useful package for the main iteration lemma in our argument.
\begin{lemma}\label{lem.newprob}
Suppose that $G$ has no $2$-torsion, $A \subset G$, $(\Omega,\P)$ is a probability space, $z$ is a $G$-valued random variable, and $Z$ is an $\mathcal{N}(G)$-valued random variable.  Then there is an extension $(\Omega',\P')$ of $(\Omega,\P)$ supporting a $G$-valued random variable $z'$, such that for any $\mathcal{N}(G)$-valued random variable $W$ on $\Omega$ we have: for all $0 \leq s \leq r$
\begin{equation}\label{eqn.kkj}
\E{\|1_A \ast m_{2^r\cdot W} - m_{2^s\cdot (z+Z)}(A)\|_{L_2(m_{2^s\cdot(z+Z)})}^2} = \E'{|m_{2^sz'+2^r\cdot W}(A)- m_{2^s\cdot (z+Z)}(A)|^2};
\end{equation}
and
\begin{equation}\label{eqn.ifr}
\|\E'{m_{z'+W}} - \E{m_{z+Z}}\|\leq \tau \text{ if $(Z(\omega),W(\omega))$ is $\tau$-closed for all $\omega \in \Omega$};
\end{equation}
and for all $0 \leq s_0 \leq r$, we have
\begin{align}\label{eqn.irf}
\sum_{s=0}^r{\E'{m_{2^s\cdot (z'+2^{r-s_0}\cdot W)}(A)^2}} & \geq \sum_{s=0}^r{\E{m_{2^s\cdot(z+Z)}(A)^2}}\\ \nonumber& \qquad \qquad +\E'{|m_{2^{s_0}z'+2^r\cdot W}(A)- m_{2^{s_0}\cdot(z+Z)}(A)|^2}-O(\tau' r)
\end{align}
if $(Z(\omega),2^rW(\omega))$ is $\tau'$-closed for all $\omega \in \Omega$.
\end{lemma}
\begin{proof}
Let $\Omega':=\{(\omega,z_*): \omega \in \Omega, z_*\in Z(\omega)\}$ and
\begin{equation*}
\P'(\{(\omega,z_*)\})=\frac{1}{|Z(\omega)|}\cdot \P(\{\omega\}) \text{ for all } (\omega,z_*) \in \Omega'.
\end{equation*}
The space $(\Omega',\P')$ is an extension of $(\Omega,\P)$ via the canonical projection $\Omega'\rightarrow \Omega$.  Let $z'(\omega,z_*):=z(\omega)+z_*$.

The first part is immediate once the definition has been unpacked (using the fact that $G$ has no $2$-torsion so that $|2^s\cdot (z(\omega)+Z(\omega))| =|Z(\omega)|$.  For the second part use Lemma \ref{lem.basicclosed} (\ref{pt.inv}) to see that
\begin{align*}
\left\|\E'{m_{z'+W}} - \E{m_{z+Z}}\right\| & = \left\|\E{\int{m_{z(\omega)+z_*+W(\omega)}dm_{Z(\omega)}(z_*)}} - \E{m_{z(\omega)+Z(\omega)}}\right\|\\
& = \left\|\E{\int{\left(\tau_w(m_{z(\omega) + Z(\omega)})- m_{z(\omega)+Z(\omega)}\right)dm_{W(\omega)}(w)}}\right\|\\
& \leq \E{\int{\left\|\tau_w(m_{z(\omega) + Z(\omega)})- m_{z(\omega)+Z(\omega)}\right\|dm_{W(\omega)}(w)}}\\
& = \E{\int{\left\|\tau_w(m_{Z(\omega)})- m_{Z(\omega)}\right\|dm_{W(\omega)}(w)}}\leq \tau.
\end{align*}
For the third part, note that for $0\leq s \leq r$ we have
\begin{align}
\label{eqn.iv} &\E'{|m_{2^sz'+2^{r+s-s_0}\cdot W}(A)- m_{2^{s}\cdot(z+Z)}(A)|^2}\\
\nonumber & \qquad \qquad \qquad= \E{\int{|m_{2^{s}\cdot (z+z_*+2^{r-s_0}\cdot W)}(A)-m_{2^{s}\cdot (z+Z)}(A)|^2dm_{Z}(z_*)}}\\
\nonumber & \qquad \qquad \qquad= \E{\bigg(\int{m_{2^{s}\cdot (z+z_*+2^{r-s_0}\cdot W)}(A)^2dm_{Z}(z_*)}}\\
\nonumber & \qquad \qquad \qquad \qquad \qquad-2m_{2^{s}\cdot (z+Z)}(A)\int{m_{2^{s}\cdot (z+z_*+2^{r-s_0}\cdot W)}(A)dm_{Z}(z_*)}\\
\nonumber &  \qquad \qquad \qquad \qquad\qquad \qquad \qquad+ m_{2^{s}\cdot (z+Z)}(A)^2\bigg)\\
\nonumber & \qquad \qquad \qquad= \E'{m_{2^{s}\cdot (z'+2^{r-s_0}\cdot W)}(A)^2} - \E'{m_{2^{s}\cdot (z+Z)}(A)^2} +O(\tau').
\end{align}
The last equality follows since
\begin{align*}
&\int{m_{2^{s}\cdot (z(\omega)+z_*+2^{r-s_0}\cdot W(\omega))}(A)dm_{Z(\omega)}(z_*)}\\
 &\qquad \qquad =\int{m_{2^s\cdot (z(\omega)+Z(\omega)+2^{r-s_0}w)}(A)dm_{W(\omega)}(w)}\\ & \qquad \qquad =\int{\tau_{2^{r-s_0}w}(m_{z(\omega)+Z(\omega)})(2^{-s}\cdot (A\cap 2^s\cdot G))dm_{W(\omega)}(w)}\\ &\qquad \qquad =m_{z(\omega)+Z(\omega)}(2^{-s}\cdot (A\cap 2^s\cdot G)) + O(\tau')=m_{2^s\cdot (z(\omega)+Z(\omega))}(A) + O(\tau'),
\end{align*}
in view of the $\tau'$-closure of $(Z(\omega),2^{r-s_0}\cdot W(\omega))$ (inherited since $2^{r-s_0}\cdot W(\omega) \subset 2^rW(\omega)$) and Lemma \ref{lem.basicclosed} (\ref{pt.inv}) and the fact that $G$ has no $2$-torsion.

Sum (\ref{eqn.iv}) over $s$ and for $s \neq s_0$ use the lower bound of $0$ for the left hand side, valid since it is an average over a square.  The result follows.
\end{proof}
With these results in hand we turn to the main technical ingredient of the whole argument which we shall iterate to get Corollary \ref{cor.iju}.
\begin{lemma}\label{lem.inductivestep}
Suppose that $G$ has no $2$-torsion; $A,S \subset G$; and $(\Omega,\P)$ is a probability space supporting a $G$-valued random variable $z$, an $\mathcal{N}(G)$-valued random variable $Z$, and an $\mathcal{S}(G)$-valued random variable $T$ such that for all $\omega \in \Omega$ we have
\begin{equation*}
(Z(\omega),T(\omega)_0)\text{ is $\tau$-closed, }  \dim T(\omega) \leq d \text{ and } \mathcal{C}^\flat(S;T(\omega)) \leq D;
\end{equation*}
and $\delta,\nu \in \left(0,\frac{1}{2}\right]$ and $r \in \N$ are parameters.  Then either $\tau^{-1}\leq (\delta^{-1}r)^{O(1)}$; or 
\begin{enumerate}
\item \label{casefirst} there is an extension $(\Omega'',\P'')$ of $(\Omega,\P)$, supporting a $G$-valued random variable $z''$, an $\mathcal{N}(G)$-valued random variable $Z''$, and a $\mathcal{S}(G)$-valued random variable $T''$ with
\begin{equation*}
\|\E''{m_{z''+Z''}} - \E{m_{z+Z}}\| \leq \tau
\end{equation*}
and
\begin{equation}\label{eqn.op}
\left(\sum_{s=0}^r{\E''{m_{2^s\cdot (z''+Z'')}(A)^2}} \right)\geq \left(\sum_{s=0}^r{\E{m_{2^s\cdot(z+Z)}(A)^2}}\right) +\delta^{O(1)},
\end{equation}
such that for all $\omega'' \in \Omega''$, $(Z''(\omega''),T''(\omega'')_0)$ is $\nu$-closed, and
\begin{equation*}
\dim T''(\omega'') \leq d +\delta^{-O(1)} \text{ and }\mathcal{C}^\flat(S;T''(\omega'')) \leq D\exp((dkr\delta^{-1}\log\nu^{-1})^{O(1)});
\end{equation*}
\item \label{casesecond} or there is an extension $(\Omega',\P')$ of $(\Omega,\P)$, supporting a $G$-valued random variable $z'$ and $\mathcal{N}(G)$-valued random variables $Z_1,\dots,Z_k$ such that
\begin{enumerate}
\item \label{acon.u1X}
\begin{equation*}
\|\E'{m_{2^sz'+Z_i}} - \E{m_{2^s\cdot(z+Z)}}\| \leq \tau
\end{equation*}
for all $1 \leq i\leq k$ and $0 \leq s \leq r$;
\item \label{acon.u1} \emph{($U_1$-uniformity)}
\begin{equation*}
\E'{|m_{2^sz'+Z_i}(A)-m_{2^s\cdot(z+Z)}(A)|^2} \leq \delta
\end{equation*}
for all $1 \leq i\leq k$ and $0 \leq s \leq r$;
\item \label{acon.u2} \emph{($U_2$-uniformity)}
\begin{equation*}
\E'{\left\|1_{A\cap (2^sz'+Z_i)} \ast (1_{A}dm_{2^sz'+Z_j})- m_{2^s\cdot(z+Z)}(A)^2\right\|_{L_2\left(m_{2^{s+1}z'+Z_i}\right)}^2}\leq \delta
\end{equation*}
for all $1 \leq i <j \leq k$ and $0 \leq s \leq r$;
\end{enumerate}
and for all $\omega' \in \Omega'$, $(Z_i(\omega'),Z_{i+1}(\omega'))$ is $\delta$-closed for all $1 \leq i < k$, $\mathcal{C}^\flat(S;Z_k(\omega')) \leq D\exp((dkr\delta^{-1}\log \nu^{-1})^{O(1)})$.
\end{enumerate}
\end{lemma}
\begin{proof}
Let $\delta_0,\delta_1,\delta_2,\delta_3$ be related constants to be optimised later.  (They will all be of the shape $\delta^{O(1)}$.)

For each $\omega \in \Omega$ we shall create sets $Z_i'(\omega)$ and naturals $m_i$ iteratively.  Let $m_0:=r$, and suppose that $m_i$ has been defined for some $0 \leq i <k$.  Apply Lemma \ref{lem.cc} to the system $2^{-m_i}T(\omega)$ to get a set $T(\omega)_{m_i+1} \subset Z_{i+1}'(\omega) \subset T(\omega)_{m_i}$ and a natural $m_{i+1}=m_i+\log_2dr\delta_0^{-1}\nu^{-1}+O(1)$ such that $(Z'_{i+1}(\omega),T(\omega)_{m_{i+1}})$ is $r^{-1}\delta_0\nu$-closed.  Although the $m_i$s may depend on $\omega$, by construction we have the universal bound $m_i=O(r+k\log dr\delta_0^{-1}\nu^{-1})$ for $0 \leq i \leq k$; and we have that
\begin{equation}\label{eqn.aclosure}
(Z'_{i}(\omega),Z'_{i+1}(\omega)) \text{ is $r^{-1}\delta_0\nu$-closed for all }1 \leq i < k \text{ and }\omega \in \Omega,
\end{equation}
and (since $2^rZ'_{i}(\omega) \subset 2^rT(\omega)_{m_0} =2^rT(\omega)_r \subset T(\omega)_0$)
\begin{equation}\label{eqn.aclosuref}
(Z(\omega),2^rZ'_{i}(\omega)) \text{ is $\tau$-closed for all }1 \leq i \leq k \text{ and }\omega \in \Omega.
\end{equation}

Apply Lemma \ref{lem.newprob} to $z$, $Z$ to get an extension $(\Omega',\P')$ of $(\Omega,\P)$.  Suppose that there is some $1 \leq i_0 \leq k$ and $0 \leq s_0 \leq r$ such that
\begin{equation}\label{eqn.assnot}
\E'{|m_{2^{s_0}z'+2^r\cdot Z_{i_0}'}(A)- m_{2^{s_0}\cdot(z+Z)}(A)|^2}\geq \delta_1.
\end{equation}
Put $Z'(\omega):=2^{r-s_0}\cdot Z_{i_0}'(\omega)$ and use conclusion (\ref{eqn.irf}) of Lemma \ref{lem.newprob} (with $W=Z_{i_0}'$ applicable in view of (\ref{eqn.aclosuref})) to see that
\begin{align*}
\left(\sum_{s=0}^r{\E'{m_{2^s\cdot (z'+Z')}(A)^2}} \right)& \geq \left(\sum_{s=0}^r{\E{m_{2^s\cdot(z+Z)}(A)^2}}\right) +\delta_1-O(\tau r).
\end{align*}
It follows that either $\tau^{-1}=O(r\delta_1^{-1})$ (and we are done) or else
\begin{align*}
\left(\sum_{s=0}^r{\E'{m_{2^s\cdot (z'+Z')}(A)^2}} \right)& \geq \left(\sum_{s=0}^r{\E{m_{2^s\cdot(z+Z)}(A)^2}}\right) +\Omega(\delta_1).
\end{align*}
Let $T'(\omega):=2^{r-s_0}\cdot (2^{-m_{i_0}}T(\omega))$.  By Lemma \ref{lem.basicclosed} (\ref{pt.dilation}), the fact that $(2^{-m_{i_0}}T(\omega))_0=T(\omega)_{m_{i_0}}$, and the construction of $Z'_{i_0}(\omega)$ we have that $(Z'(\omega),T'(\omega)_0)$ is $\delta_0$-closed.  Moreover, by Lemma \ref{lem.new} (\ref{pt.mult}) and then Lemma \ref{lem.new} (\ref{pt.szy})
\begin{align}
\label{eqn.szest}\mathcal{C}^\flat(S;T'(\omega)_0)&=\mathcal{C}^\flat(S;2^{r-s_0}\cdot T(\omega)_{m_{i_0}})\\ \nonumber & \leq \exp(O(rd))\mathcal{C}^\flat(S;T(\omega)_{m_{i_0}})\\ \nonumber&  \leq \exp(O(rd+m_{i_0}))\mathcal{C}^\flat(S;T(\omega)_0) \leq D\exp((dkr\delta_0^{-1}\log \nu^{-1})^{O(1)}).
\end{align} 
Finally, use conclusion (\ref{eqn.ifr}) of Lemma \ref{lem.newprob} with $W=2^{r-s_0}\cdot Z_{i_0}'$ (applicable in view of (\ref{eqn.aclosuref})) and the fact that
\begin{equation*}
2^{r-s_0}\cdot Z_{i_0}'(\omega) \subset 2^{r-s_0}Z_{i_0}'(\omega) \subset 2^rZ_{i_0}'(\omega) \subset 2^rT(\omega)_{m_{i_0-1}}\subset 2^rT(\omega)_{m_0} = 2^rT(\omega)_r \subset T(\omega)_0
\end{equation*}
to see that
\begin{equation*}
\|\E'{m_{z'+Z'}} - \E{m_{z+Z}}\| \leq \tau.
\end{equation*}
Set $(\Omega'',\P''):=(\Omega',\P')$, $Z'':=Z'$, $z'':=z'$ and $T'':=T'$, and we are in case (\ref{casefirst}) of the lemma.

In view of this we assume that there is no $i_0$ or $s_0$ (in the given ranges) such that (\ref{eqn.assnot}) holds.  Put $Z_i:=2^r\cdot Z_i'$ for all $1\leq i \leq k$.  Then for $1 \leq i \leq k$ and $0 \leq s \leq r$ we have
\begin{align*}
\|\E'{m_{2^sz'+Z_i}} - \E{m_{2^s\cdot(z+Z)}}\| & = \|\E'{m_{2^sz'+2^r\cdot Z_i'}} - \E{m_{2^s\cdot(z+Z)}}\|\\
& = \|\E'{m_{z'+2^{r-s}\cdot Z_i'}} - \E{m_{z+Z}}\|\leq \tau,
\end{align*}
by (\ref{eqn.ifr}) with $W=2^{r-s}\cdot Z_i'$ since $2^{r-s}\cdot Z_i'(\omega) \subset 2^rZ_i'(\omega)\subset 2^rT(\omega)_{m_0} \subset T(\omega)_0$; we have established the conclusion (\ref{acon.u1X}).

Now, by (\ref{eqn.aclosure}) and Lemma \ref{lem.basicclosed} (\ref{pt.dilation}) we have that
\begin{equation}\label{eqn.zinv}
(Z_i(\omega),Z_{i+1}(\omega)) \text{ is $\delta_0$-closed for all }1 \leq i <k.
\end{equation}
Since $Z_k'(\omega)\supset T(\omega)_{m_k+1}$, Lemma \ref{lem.cnum} (\ref{pt.order}) and almost exactly the same argument as in (\ref{eqn.szest}) shows that
\begin{equation*}
\mathcal{C}^\flat(S;Z_k'(\omega)) \leq D\exp((dkr\delta_0^{-1}\log\nu^{-1})^{O(1)}).
\end{equation*}\
Since (\ref{eqn.assnot}) does not hold for any $1 \leq i \leq k$ or $0 \leq s\leq r$ we have that
\begin{equation*}
\E'{|m_{2^{s}z'+Z_i}(A)- m_{2^{s}\cdot(z+Z)}(A)|^2}<\delta_1 \text{ for all }1 \leq i \leq k, 0 \leq s \leq r.
\end{equation*}
Conclusion (\ref{acon.u1}) follows.

Suppose that for all $1 \leq i <j \leq k$ and $0 \leq s \leq r$ we have
\begin{equation*}
\E'{\|1_{A\cap (2^sz'+Z_i)} \ast (1_A dm_{2^sz'+Z_j}) - m_{2^s\cdot (z+Z)}(A)\|_{L_2(m_{2^{s+1}z'+Z_i})}^2} < \delta_2.
\end{equation*}
Then conclusion (\ref{acon.u2}) follows, and we are in case (\ref{casesecond}) of the lemma.  Thus we assume not so that there are elements $1 \leq i_1 <j_1 \leq k$ and $0 \leq s_1 \leq r$ such that
\begin{equation}\label{eqn.kymss}
\E'{\|1_{A\cap (2^{s_1}z'+Z_{i_1})} \ast (1_A dm_{2^{s_1}z'+Z_{j_1}}) - m_{2^{s_1}\cdot (z+Z)}(A)\|_{L_2(m_{2^{s_1+1}z'+Z_{i_1}})}^2} \geq \delta_2.
\end{equation}
For each $\omega \in \Omega$ put $T'(\omega):=2^r\cdot (2^{-{m_k}}T(\omega))$ so that Lemma \ref{lem.basicclosed} (\ref{pt.dilation}) tells us that $(Z_k(\omega),T'(\omega)_0)$ is $\delta_0$-closed.

For each $\omega' \in \Omega'$ apply Lemma \ref{lem.increment} with $\alpha= m_{2^{s_1}\cdot(z(\omega')+Z(\omega'))}(A)$ and both free parameters equal to $\delta_3$ (for reasons which will become clear) to $(Z_{i_1}(\omega'),Z_{j_1}(\omega'))$ and $(Z_{j_1}(\omega'),T'(\omega')_0)$ (which are both $\delta_0$-closed, the former by (\ref{eqn.zinv})), and the set $A-2^{s_1}z'(\omega')$.

Out of the lemma we get a Bohr system $B''(\omega')$ with $\dim B''(\omega') \leq \delta_3^{-O(1)}$ and $\mathcal{C}^\flat(G;B''(\omega')_0) \leq \exp(\delta_3^{-O(1)})$ such that for any $W''(\omega') \subset T'(\omega')_0\cap B''(\omega')_0$ we have
\begin{align*}
&\left\|1_{(A-2^{s_1}z'(\omega'))} \ast m_{W''(\omega')} - m_{2^{s_1}\cdot(z(\omega')+Z(\omega'))}(A)\right\|_{L_2\left(m_{Z_{i_1}(\omega')}\right)}^2\\
& \qquad  \geq \left\|1_{(A-2^{s_1}z'(\omega'))\cap Z_{i_1}(\omega')} \ast (1_{A-2^{s_1}z'(\omega')} dm_{Z_{j_1}(\omega')}) - m_{2^{s_1}\cdot (z(\omega')+Z(\omega'))}(A)^2\right\|_{L_2\left(m_{Z_{i_1}(\omega')}\right)}^2\\
&\qquad \qquad \qquad  - O\left( \delta_3 + \delta_3^{-2}\delta_0 + \delta_3^{-2}|m_{2^{s_1}z'(\omega')+Z_{i_1}(\omega')}(A)-m_{2^{s_1}\cdot (z(\omega')+Z(\omega'))}(A)|\right)\\
&\qquad \qquad \qquad \qquad \qquad \qquad - O\left(\delta_3^{-2}|m_{2^{s_1}z'(\omega')+Z_{j_1}(\omega')}(A)-m_{2^{s_1}\cdot (z(\omega')+Z(\omega'))}(A)|\right)\\
& \qquad  = \big\|1_{A\cap (2^{s_1}z'(\omega')+Z_{i_1}(\omega'))} \ast (1_{A} dm_{2^{s_1}z'(\omega')+Z_{j_1}(\omega')})\\
&\qquad \qquad \qquad \qquad \qquad \qquad \qquad \qquad \qquad - m_{2^{s_1}\cdot (z(\omega')+Z(\omega'))}(A)^2\big\|_{L_2\left(m_{2^{s_1+1}z'(\omega')+Z_{i_1}(\omega')}\right)}^2\\
&\qquad \qquad \qquad  - O\left( \delta_3 + \delta_3^{-2}\delta_0 +  \delta_3^{-2}|m_{2^{s_1}z'(\omega')+Z_{i_1}(\omega')}(A)-m_{2^{s_1}\cdot (z(\omega')+Z(\omega'))}(A)|\right)\\
&\qquad \qquad \qquad \qquad \qquad \qquad  - O\left( \delta_3^{-2}|m_{2^{s_1}z'(\omega')+Z_{j_1}(\omega')}(A)-m_{2^{s_1}\cdot (z(\omega')+Z(\omega'))}(A)|\right).
\end{align*}
Apply Lemma \ref{lem.cc} to the system $2^{-2r}(T'(\omega')\wedge B''(\omega'))$ (which has dimension $d+\delta_3^{-O(1)}$ by (Lemma \ref{lem.new} (\ref{pt.intdim}) and (\ref{pt.szy})) to get a set $Z''(\omega') \subset (T'(\omega')\wedge B''(\omega'))_{2r}$ and a natural $m'=O(\log_2d\delta_3^{-1}r\delta_0^{-1})$ such that $(Z''(\omega'),(T'(\omega')\wedge B''(\omega'))_{m'+2r})$ is $r^{-1}\delta_0$-closed.  Put $T''(\omega'):=2^r\cdot (2^{-(m'+r)}(T'(\omega')\wedge B''(\omega')))$ which has dimension $d+\delta_3^{-O(1)}$, and
\begin{equation*}
\mathcal{C}^\flat(S;T''(\omega')_0) \leq D\exp((dkr\delta_3^{-1}\delta_0^{-1}\log\nu^{-1})^{O(1)}).
\end{equation*}
Taking $W''(\omega'):=2^r\cdot Z''(\omega')$ (which is contained in $T'(\omega')_0\cap B''(\omega')_0$ by design), and averaging against $\P'$ we have
\begin{align*}
&\E'{\|1_{A-2^{s_1}z'} \ast m_{2^r\cdot Z''} - m_{2^{s_1}\cdot (z+Z)}(A)\|_{L_2\left(m_{Z_{i_1}}\right)}^2}\\
& \qquad  \geq \delta_2  - O\left( \delta_3 + \delta_3^{-2}\delta_0 + \delta_3^{-2}\E'{|m_{2^{s_1}z'+Z_{i_1}}(A)-m_{2^{s_1}\cdot (z+Z)}(A)|}\right)\\
&\qquad \qquad \qquad  - O\left( \delta_3^{-2}\E'{|m_{2^{s_1}z'+Z_{j_1}}(A)-m_{2^{s_1}\cdot (z+Z)}(A)|}\right)\\
& \qquad \geq \delta_2 - O\left(\delta_3 +\delta_3^{-2}\delta_0 + \delta_3^{-2}\left(\E{|m_{2^{s_1}z'+Z_{i_1}}(A)-m_{2^{s_1}\cdot (z+Z)}(A)|^2}\right)^{\frac{1}{2}}\right)\\
&\qquad \qquad \qquad  - O\left( \delta_3^{-2}\left(\E'{|m_{2^{s_1}z'+Z_{j_1}}(A)-m_{2^{s_1}\cdot (z+Z)}(A)|^2}\right)^{\frac{1}{2}}\right)\\
& \qquad \geq \delta_2 - O(\delta_3 + \delta_3^{-2}\delta_0 + \delta_3^{-2}\delta_1^{\frac{1}{2}})
\end{align*}
by (\ref{eqn.kymss}), the Cauchy-Schwarz inequality, and the fact that (\ref{eqn.assnot}) does not hold for $i_0=i_1$ and $s_0=s_1$, or $i_0=j_1$ and $s_0=s_1$.

Again, since (\ref{eqn.assnot}) does not hold for $s_0=s_1$ and $i_0=i_1$, we have from the Cauchy-Schwarz inequality that
\begin{align*}
&\E'{\|1_{A} \ast m_{2^r\cdot Z''} - m_{2^{s_1}z'+Z_{i_1}}(A)\|_{L_2\left(m_{2^{s_1}z'+Z_{i_1}}\right)}^2}\\
& \qquad \qquad \geq \E'{\|1_{A} \ast m_{2^r\cdot Z''} - m_{2^{s_1}\cdot(z+Z)}(A)\|_{L_2\left(m_{2^{s_1}z'+Z_{i_1}}\right)}^2}\\
&  \qquad \qquad  \qquad \qquad -O(\E'{|m_{2^{s_1}\cdot(z+Z)}(A)-m_{2^{s_1}z'+Z_{i_1}}(A)|})\\
& \qquad \qquad \geq \E'{\|1_{A-2^{s_1}z'} \ast m_{2^r\cdot Z''} - m_{2^{s_1}\cdot (z+Z)}(A)\|_{L_2\left(m_{Z_{i_1}}\right)}^2} -O(\delta^{-\frac{1}{2}}).
\end{align*}
Combining all this tells us that
\begin{equation*}
\E'{\|1_{A} \ast m_{2^r\cdot Z''} - m_{2^{s_1}z'+Z_{i_1}}(A)\|_{L_2(m_{2^{s_1}z'+Z_{i_1}})}^2} \geq  \delta_2 - O(\delta_3 + \delta_3^{-2}\delta_0 + \delta_3^{-2}\delta_1^{\frac{1}{2}}).
\end{equation*}
Recall that $Z_{i_1}=2^r\cdot Z_{i_1}'$ and apply Lemma \ref{lem.newprob} to $(\Omega',\P')$, $z'$ and $2^{r-s_1}\cdot Z_{i_1}'$ to get $(\Omega'',\P'')$ and $z''$.  (\ref{eqn.kkj}) combined with the above tells us that
\begin{equation*}
\E''{|m_{2^{s_1}z''+2^r\cdot Z''}(A)-m_{2^{s_1}\cdot (z'+2^{r-s_1}\cdot Z_{i_1}')}(A)|^2}  \geq  \delta_2 - O(\delta_3 + \delta_3^{-2}\delta_0 + \delta_3^{-2}\delta_1^{\frac{1}{2}}).
\end{equation*}
Since $(2^{r-s_1}\cdot Z_{i_1}'(\omega'),2^rZ''(\omega'))$ is $r^{-1}\delta_0$-closed (since $2^{r-s_1}\cdot (2^rZ''(\omega')) \subset T'(\omega')_0$ by design), (\ref{eqn.irf}) for $s_0=s_1$ tells us that
\begin{equation*}
\sum_{s=0}^r{\E''{m_{2^s\cdot (z''+2^{r-s_1}\cdot Z'')}(A)^2}} \geq \sum_{s=0}^r{\E'{m_{2^s\cdot (z'+2^{r-s_1}\cdot Z_{i_1}')}(A)^2}} +\delta_2 - O(\delta_3 + \delta_3^{-2}\delta_0 + \delta_3^{-2}\delta_1^{\frac{1}{2}}).
\end{equation*}
On the other hand from (\ref{eqn.irf}) (when we applied Lemma \ref{lem.newprob} to get $(\Omega',\P')$) and the fact that $(Z(\omega),2^rZ_{i_1}'(\omega))$ is $\tau$-closed (\ref{eqn.aclosuref}) we also have
\begin{equation*}
\sum_{s=0}^r{\E'{m_{2^s\cdot (z'+2^{r-s_1}\cdot Z_{i_1}')}(A)^2}} \geq \sum_{s=0}^r{\E{m_{2^s\cdot (z+Z)}(A)^2}} - O(\tau r).
\end{equation*}
Hence (either $\tau^{-1} =O(\delta_2^{-1}r^{-1})$) or else
\begin{equation*}
\sum_{s=0}^r{\E''{m_{2^s\cdot (z''+2^{r-s_1}\cdot Z'')}(A)^2}} \geq  \sum_{s=0}^r{\E{m_{2^s\cdot (z+Z)}(A)^2}} +\delta_2 - O(\delta_3 + \delta_3^{-2}\delta_0 + \delta_3^{-2}\delta_1^{\frac{1}{2}}).
\end{equation*}
Taking $\delta_2=\delta$, $\delta_3=c\delta$, $\delta_1=c\delta^6$ and $\delta_0=c\delta^3$ for some sufficiently small $c$ gives the result and we are in case (\ref{casefirst}) of the lemma.
\end{proof}

\begin{corollary*}[Corollary \ref{cor.iju}]Suppose that $G$ has no $2$-torsion; $A,S \subset G$; $(\Omega,\P)$ is a probability space supporting a $G$-valued random variable $z$, an $\mathcal{N}(G)$-valued random variable $Z$ and an $\mathcal{S}(G)$-valued random variable $T$ such that for all $\omega \in \Omega$,
\begin{equation*}
\dim T(\omega) \leq d \text{ and } \mathcal{C}^\flat(S;T(\omega)) \leq D
\end{equation*}
and $(Z(\omega),T(\omega)_0)$ is $\tau$-closed; and $\delta \in (0,1]$ and $r \in \N$ are parameters.  Then either $\tau^{-1}\leq (\delta^{-1}r)^{O(1)}$; or there is a probability space $(\Omega',\P')$ extending $(\Omega,\P)$, supporting a $G$-valued random variable $z'$ and $\mathcal{N}(G)$-valued random variable $Z'$ with
\begin{equation*}
\|\E'{m_{z'+Z'}} - \E{m_{z+Z}}\| \leq\delta,
\end{equation*}
and a further extension $(\Omega'',\P'')$ of $(\Omega',\P')$, supporting a $G$-valued random variable $z''$ and $\mathcal{N}(G)$-valued random variables $Z_1,\dots,Z_k$ such that
\begin{enumerate}
\item 
\begin{equation*}
\|\E''{m_{2^sz''+Z_i}} - \E'{m_{2^s\cdot(z'+Z')}}\| \leq \delta
\end{equation*}
for all $1 \leq i\leq k$ and $0 \leq s \leq r$;
\item  \emph{($U_1$-uniformity)}
\begin{equation*}
\E''{|m_{2^sz''+Z_i}(A)-m_{2^s\cdot(z'+Z')}(A)|^2} \leq \delta
\end{equation*}
for all $1 \leq i\leq k$ and $0 \leq s \leq r$;
\item  \emph{($U_2$-uniformity)}
\begin{equation*}
\E''{\left\|1_{A\cap (2^sz''+Z_i)} \ast (1_{A}dm_{2^sz''+Z_j})- m_{2^s\cdot(z'+Z')}(A)^2\right\|_{L_2\left(m_{2^{s+1}z''+Z_i}\right)}^2}\leq \delta
\end{equation*}
for all $1 \leq i <j \leq k$ and $0 \leq s \leq r$;
\end{enumerate}
and for all $\omega'' \in \Omega''$, $(Z_i(\omega''),Z_{i+1}(\omega''))$ is $\delta$-closed for all $1 \leq i < k$, and $\mathcal{C}^\flat(S;Z_k(\omega'')) \leq D\exp((dkr\delta^{-1})^{O(1)})$.
Suppose that $G$ has no $2$-torsion; $A,S \subset G$; $(\Omega,\P)$ is a probability space supporting a $G$-valued random variable $z$, an $\mathcal{N}(G)$-valued random variable $Z$ and an $\mathcal{S}(G)$-valued random variable $T$ such that for all $\omega \in \Omega$,
\begin{equation*}
\dim T(\omega) \leq d \text{ and } \mathcal{C}^\flat(S;T(\omega)) \leq D
\end{equation*}
and $(Z(\omega),T(\omega)_0)$ is $\tau$-closed; and $\delta \in (0,1]$ and $r \in \N$ are parameters.  Then either $\tau^{-1}\leq (\delta^{-1}r)^{O(1)}$; or there is a probability space $(\Omega',\P')$ extending $(\Omega,\P)$, supporting a $G$-valued random variable $z'$ and $\mathcal{N}(G)$-valued random variable $Z'$ with
\begin{equation*}
\|\E'{m_{z'+Z'}} - \E{m_{z+Z}}\| \leq\delta,
\end{equation*}
and a further extension $(\Omega'',\P'')$ of $(\Omega',\P')$, supporting a $G$-valued random variable $z''$ and $\mathcal{N}(G)$-valued random variables $Z_1,\dots,Z_k$ such that
\begin{enumerate}
\item 
\begin{equation*}
\|\E''{m_{2^sz''+Z_i}} - \E'{m_{2^s\cdot(z'+Z')}}\| \leq \delta
\end{equation*}
for all $1 \leq i\leq k$ and $0 \leq s \leq r$;
\item  \emph{($U_1$-uniformity)}
\begin{equation*}
\E''{|m_{2^sz''+Z_i}(A)-m_{2^s\cdot(z'+Z')}(A)|^2} \leq \delta
\end{equation*}
for all $1 \leq i\leq k$ and $0 \leq s \leq r$;
\item  \emph{($U_2$-uniformity)}
\begin{equation*}
\E''{\left\|1_{A\cap (2^sz''+Z_i)} \ast (1_{A}dm_{2^sz''+Z_j})- m_{2^s\cdot(z'+Z')}(A)^2\right\|_{L_2\left(m_{2^{s+1}z''+Z_i}\right)}^2}\leq \delta
\end{equation*}
for all $1 \leq i <j \leq k$ and $0 \leq s \leq r$;
\end{enumerate}
and for all $\omega'' \in \Omega''$, $(Z_i(\omega''),Z_{i+1}(\omega''))$ is $\delta$-closed for all $1 \leq i < k$, and $\mathcal{C}^\flat(S;Z_k(\omega'')) \leq D\exp((dkr\delta^{-1})^{O(1)})$.
\end{corollary*}
\begin{proof}
Let $\tau_0^{-1}=(r\delta^{-1})^{O(1)}$ be the function in the first conclusion of Lemma \ref{lem.inductivestep} applied with parameters $r$ and $\delta$.  If $\tau^{-1} \leq 2\delta^{-1}$ then terminate with $\tau^{-1}=(r\delta^{-1})^{O(1)}$, so assume not.  Let $\delta_0=\delta^{O(1)}$ be the lower bound in (\ref{eqn.op}) (when that lemma is applied with parameter $\delta$), and let $\nu_0^{-1}=(r\delta^{-1})^{O(1)}$ be such that $\nu_0<\tau_0$ and $\nu_0(r+1)\lceil \delta_0^{-1}\rceil + \tau \leq \delta$ which is possible since $\tau \leq \frac{1}{2}\delta$.

We proceed inductively to define $\Omega^{(i)}$, $\P^{(i)}$ such that $(\Omega^{(i)},\P^{(i)})$ is a probability space, $z^{(i)}$ is a $G$-valued random variable, $Z^{(i)}$ is an $\mathcal{N}(G)$-valued random variable, and $T^{(i)}$ is an $\mathcal{S}(G)$-valued random variable such that for all $\omega^{(i)} \in \Omega^{(i)}$ we have
\begin{equation*}
(Z^{(i)}(\omega^{(i)}),T^{(i)}(\omega^{(i)})) \text{ is $\tau$-closed if $i=0$, and $\nu_0$-closed if $i>0$},
\end{equation*}
\begin{equation*}
\dim T^{(i)}(\omega^{(i)})\leq d+i\delta^{-O(1)},
\end{equation*}
\begin{equation*}
\mathcal{C}^\flat(S;T^{(i)}(\omega^{(i)})) \leq D\exp(i(dkr\delta^{-1}\log\nu_0^{-1})^{O(1)}),
\end{equation*}
\begin{equation}\label{eqn.allow}
\|\E^{(i)}{m_{z^{(i)}+Z^{(i)}}} - \E{m_{z+Z}}\| \leq \begin{cases}0 & \text{ if } i=0\\ \tau + (i-1)\nu_0 & \text{ if }i>0\end{cases},
\end{equation}
and
\begin{equation}\label{eqn.summax}
\sum_{s=0}^r{\E^{(i)}{m_{2^s\cdot (z^{(i)} + Z^{(i)})}(A)^2}} \geq i\delta_0.
\end{equation}
We initialise with $\Omega^{(0)}:=\Omega$, $\P^{(0)}:=\P$, $z^{(0)}:=z$, $Z^{(0)}:=Z$, and $T^{(0)}:=T$, which satisfies the above requirements trivially.  At stage $i$ apply Lemma \ref{lem.inductivestep} to the space $(\Omega^{(i)},\P^{(i)})$; random variables $z^{(i)}$, $Z^{(i)}$ and $T^{(i)}$; parameter $\nu_0$ in place of $\nu$; $\nu_0$ or $\tau$ in place of $\tau$ (as named in Lemma \ref{lem.inductivestep}) depending on whether $i=0$ or $i>0$; and $\delta$ and $r$ as given.

Since $\nu_0^{-1}>\tau_0^{-1}$ we are not in the first case of the lemma.  (And if $i=0$ we can assume we are not in the first case or else we are in the $\tau$ large conclusion of the corollary.)  We shall terminate if in case (\ref{casesecond}) of the lemma, so assume not.  It follows we are in case (\ref{casefirst}) of the lemma which gives us an extension  $(\Omega^{(i+1)},\P^{(i+1)})$ of $(\Omega^{(i)},\P^{(i)})$ and random variables $z^{(i+1)}$, $Z^{(i+1)}$ and $T^{(i+1)}$ with (\ref{eqn.allow}) being a result of the triangle inequality.

In view of (\ref{eqn.summax}) this iteration cannot proceed for more than $(r+1)\lceil \delta_0^{-1}\rceil$ steps at which point we are in case (\ref{casesecond}) of Lemma \ref{lem.inductivestep}.  The conclusion follows from the triangle inequality again.
\end{proof}

\section{Counting}\label{sec.k}

In this section we prove the following which is the analogue of the model Lemma \ref{lem.ct}.  The key feature is that bound on $\epsilon$ in the third of the four possible conclusions only depends on $k$.  If we were prepared to admit $\alpha$-dependence then the uniformity of hypothesis (\ref{hyp.3}) would not be necessary.

\begin{lemma*}[Lemma \ref{lem.c}]
Suppose that $A,X \subset G$; $z_0 \in G$; $(Z_i,Z_{i+1})$ is $\tau$-closed for all $1 \leq i < k$; and
\begin{enumerate}
\item $|m_{Z_i}(A-z_0) -\alpha|\leq \tau$ for $1 \leq i \leq k$;
\item $m_{Z_i}(X-2z_0) \leq \epsilon$ for all $1 \leq i < k$;
\item 
\begin{equation*}
\left\|1_{(A-z_0)\cap Z_i} \ast (1_{A-z_0}dm_{Z_j}) - \alpha^2\right\|_{L_2(m_{Z_i})}^2 \leq \delta \text{ for all }1\leq i <j \leq k.
\end{equation*}
\end{enumerate}
Then either $\delta^{-1}=O(k^2\alpha^{-4})$; or $\tau^{-1} = O(k\alpha^{-1})$; or $\epsilon^{-1}=O(k^2)$; or
\begin{equation}\label{eqn.integral}
\int{\left(\prod_{i<j}{1_{(Z_i+2z_0)\setminus X}(z_i+z_j+2z_0)1_{Z_i}(z_i+z_j)}\right)\prod_{i=1}^k{1_A(z_i+z_0)dm_{Z_i}(z_i)}} =\Omega\left(\alpha^k\right). 
\end{equation}
\end{lemma*}
\begin{proof}
Replacing $A$ by $A+z_0$ and $X$ by $X+2z_0$ we may assume that $z_0=0_G$.  Recall the inequality
\begin{equation*}
\prod_{1 \leq i'<j' \leq k}{(1-x_{i'j'})} \geq 1-\sum_{1 \leq i'<j' \leq k}{x_{i'j'}} \text{ whenever }0 \leq x_{i'j'} \leq 1 \text{ for all }1 \leq i'<j'\leq k;
\end{equation*}
this is what we call the pigeonhole principle. Write $I$ for the integral in (\ref{eqn.integral}).  Then using the stated pigeonhole principle and integrating we have
\begin{align*}
I & \geq \int{\left(\prod_{1 \leq i<j \leq k}{1_{Z_i}(z_i+z_j)}\right)\prod_{i=1}^k{1_A(z_i)dm_{Z_i}(z_i)}} \\
& \qquad -\sum_{1 \leq i'<j' \leq k}{\int{1_{X\cap Z_{i'}}(z_{i'}+z_{j'})\left(\prod_{\substack{1 \leq i<j\leq k\\ (i,j) \neq (i',j')}}{1_{ Z_i}(z_i+z_j)}\right)\prod_{i=1}^k{1_A(z_i)dm_{Z_i}(z_i)}}}.
\end{align*}
Since the $Z_i$s are nested, for fixed $1\leq i <k$, we have
\begin{equation*}
\prod_{i<j \leq k}{1_{Z_i}(z_i+z_j)} \geq 1_{Z_i^-}(z_i) \text{ for all }z_{i+1} \in Z_{i+1},\dots,z_k\in Z_k.
\end{equation*}
From this and hypothesis (\ref{hyp.1}) we conclude that
\begin{align*}
I & \geq \prod_{i=1}^k{m_{Z_i}(A\cap Z_i^-)} - \sum_{1 \leq i' < j' \leq k}{\int{1_{X\cap Z_{i'}}(z_{i'}+z_{j'})\prod_{i=1}^k{1_A(z_i)dm_{Z_i}(z_i)}}}\\
& \geq (\alpha-2\tau)^k  -(\alpha+\tau)^{k-2}\sum_{1 \leq i' < j' \leq k}{\int{1_{X\cap Z_{i'}}(z_{i'}+z_{j'})1_A(z_{i'})dm_{Z_{i'}}(z_{i'})1_A(z_{j'})dm_{Z_{j'}}(z_{j'})}}\\
& = (\alpha-2\tau)^k-(\alpha+\tau)^{k-2}\sum_{1 \leq i' < j' \leq k}{\langle 1_{X},1_{A \cap Z_{i'}}\ast (1_A dm_{Z_{j'}})\rangle_{L_2(m_{Z_{i'}})}}.
\end{align*}
Now for fixed $1 \leq i'<j'\leq k$, the Cauchy-Schwarz inequality and hypothesis (\ref{hyp.3}) tell us that
\begin{equation*}
\left|\langle 1_{X},1_{A \cap Z_{i'}}\ast (1_A dm_{Z_{j'}})\rangle_{L_2(m_{Z_{i'}})} -\langle 1_{X},\alpha^2\rangle_{L_2(m_{Z_{i'}})}\right| \leq \delta^{\frac{1}{2}}m_{Z_{i'}}(X)^{\frac{1}{2}},
\end{equation*}
and so by (\ref{hyp.2}) we see that
\begin{equation*}
\langle 1_{X},1_{A \cap Z_{i'}}\ast (1_A dm_{Z_{j'}})\rangle_{L_2(m_{Z_{i'}})} \leq \langle 1_{X},\alpha^2\rangle_{L_2(m_{Z_{i'}})}+(\delta m_{Z_{i'}}(X))^{\frac{1}{2}} < \epsilon\alpha^2 + (\epsilon\delta)^{\frac{1}{2}}.
\end{equation*}
Combining all this we get that
\begin{equation*}
I \geq (\alpha-2\tau)^k - \binom{k}{2}(\alpha + \tau)^{k-2}(\epsilon\alpha^2 + (\epsilon \delta)^{\frac{1}{2}}).
\end{equation*}
It follows that either $\delta^{-1}= O(k^2\alpha^{-4})$; or $\tau^{-1}=O(k\alpha^{-1})$; or $\epsilon^{-1}=O(k^2)$; and if none of these holds then $I=\Omega(\alpha^k)$ as claimed.  The result is proved.
\end{proof}

\section*{Acknowledgements}  The author should like to thank the three referees both for their comments which very much improved the paper, and also their care in reading the paper which given its technical nature was no small ask.

\bibliographystyle{halpha}

\bibliography{references}

\end{document}